\documentclass[a4paper,11pt, twoside]{article}
\usepackage{amsfonts}
\usepackage{CJK}
\usepackage{amsthm}
\usepackage{amsmath}
\usepackage{amssymb}
\usepackage{graphicx}
\usepackage{epstopdf}
\usepackage{layout}
\usepackage{mathrsfs}
\usepackage{epsfig}
\usepackage{indentfirst,latexsym,bm}
\usepackage{pifont}
\usepackage{txfonts}
\usepackage{xcolor}
\usepackage{fancyhdr}

\newcommand{\B}{{\mathbb{B}}}
\newcommand{\R}{{\mathbb{R}}}
\newcommand{\N}{{\mathbb{N}}}

\newcommand{\PP}{{\mathbb{B}}}

\newcommand{\ha}{{\mathcal{A}}}
\newcommand{\hc}{{\mathcal{C}}}
\newcommand{\hf}{{\mathcal{F}}}
\newcommand{\hh}{{\mathcal{H}}}

\newcommand\ol{\overline}
\newcommand\lc{\langle}
\newcommand\rc{\rangle}

\newtheorem{theorem}{Theorem}[section]
\newtheorem{lemma}[theorem]{Lemma}

\newtheorem{coro}[theorem]{Corollary}

\def\bel{\begin{equation}\label}
\def\eeq{\end{equation}}
\def\pd{{\partial}}
\def\dd{{{\rm d}}}

\def\ga1{{1, \gamma}}

\def\xy{{(x, y, z)}}
\def\w2{W^{2, \infty}}
\def\ww{W^{1, \infty}}

\begin{document}

\renewcommand{\theequation}{\thesection.\arabic{equation}}

\title{Structural Stability of Supersonic Contact Discontinuities in Three-Dimensional Compressible
Steady Flows\footnote{This work is partially supported
by NNSF of China under the grants 10971134, 11031001 and 91230102.}}

\author{Ya-Guang Wang\footnote{Department of Mathematics, and MOE-LSC, Shanghai Jiao Tong University,
Shanghai 200240, China (\textbf{ygwang@sjtu.edu.cn}).} 
~and~ Fang Yu\footnote{Department of Mathematics, Shanghai Jiao Tong University, Shanghai 200240, China; Present address: Department of Mathematics, Pennsylvania State University, University Park, PA 16802, USA (\textbf{fangyuvi@gmail.com}
).}}

\date{}

\maketitle

\begin{abstract}
In this paper, we study the structurally nonlinear stability of supersonic contact discontinuities in three-dimensional compressible isentropic steady flows. Based on the weakly linear stability result and the  $L^2$-estimates obtained in \cite{WangYu1},  for the linearized problems of three-dimensional compressible isentropic steady equations at a supersonic contact discontinuity satisfying certain stability conditions, we first  derive tame estimates of solutions to the linearized problem in higher order norms by exploring the behavior of vorticities. Since the supersonic contact discontinuities are only weakly linearly stable, so the tame estimates of solutions to the linearized problems have loss of regularity with respect to both of background states and initial data, so to use the tame estimates to study the nonlinear problem 
we adapt the Nash-Moser-H\"ormander iteration scheme to conclude that supersonic contact discontinuities in three-dimensional compressible steady flows satisfying the stability conditions (\cite{WangYu1}) are structurally nonlinearly stable at least locally in space.

\end{abstract}

\smallskip


\tableofcontents
\section{Introduction}
\setcounter{equation}{0}

Based on the conservation of density and momentum, the steady compressible isentropic inviscid flows
in three space variables can be described by the following equations,
\bel{euler}
\begin{cases}
\frac{\partial}{\partial x}(\rho u)+\frac{\partial}{\partial y}(\rho v)+
\frac{\partial}{\partial y}(\rho w)=0\\
\frac{\partial}{\partial x}(\rho u^2+p(\rho))+\frac{\partial}{\partial y}(\rho uv)+
\frac{\partial}{\partial y}(\rho uw)=0\\
\frac{\partial}{\partial x}(\rho uv)+\frac{\partial}{\partial y}(\rho v^2+p(\rho))+
\frac{\partial}{\partial y}(\rho vw)=0\\
\frac{\partial}{\partial x}(\rho uw)+\frac{\partial}{\partial y}(\rho vw)+
\frac{\partial}{\partial y}(\rho w^2+p(\rho))=0
\end{cases}\eeq
where $\rho$, $p=p(\rho)$ and
$(u, v, w) \in \R^3$ denote the density, pressure and velocity  of the fluid respectively, with $p'(\rho)>0$ for $\rho>0$. It is an important model in gas dynamics, aeronautics and astronautics.

Set $U=(u,v,w,p)^T$. Obviously, the system \eqref{euler} can be rewritten as the following symmetric form \bel{mfeuler} A_1(U)\pd_xU+A_2(U)\pd_yU+A_3(U)\pd_zU=0,
\eeq where
\begin{equation}\label{A}
  A_1(U)=
  \begin{pmatrix}
    \rho u&0&0&1\\
    0&\rho u&0&0\\
    0&0&\rho u&0\\
    1&0&0&\frac{u}{\rho c^2}
  \end{pmatrix},
 \quad   A_2(U)=
  \begin{pmatrix}
    \rho v&0&0&0\\
    0&\rho v&0&1\\
    0&0&\rho v&0\\
    0&1&0&\frac{v}{\rho c^2}
  \end{pmatrix},
\quad    A_3(U)=
  \begin{pmatrix}
    \rho w&0&0&0\\
    0&\rho w&0&0\\
    0&0&\rho w&1\\
    0&0&1&\frac{w}{\rho c^2}
  \end{pmatrix},
\end{equation}
with
$c=\sqrt{p'(\rho)}$ being the sonic speed.
When the velocity in the $x$-direction is supersonic, i.e. $u>c$, the coefficient matrix $A_1(U)$
is positively definite, then the system \eqref{mfeuler} is symmetric hyperbolic with $x$ being regarded as the time-like 
direction. 

As shown in the monographes \cite{CF, majda, CFe3} etc., it is an important and challenging field to study the propagation, interaction and stability of elementary waves such as the shocks, rarefaction waves and contact discontinuities in  quasilinear hyperbolic conservation laws. The stability of
shocks and rarefaction waves in multi-dimensional gas dynamics has been studied by Majda
\cite{majda}, Metivier et al. \cite{met1, met2}, and Alinhac
\cite{ali}, 
Contact discontinuities occur ubiquitously, such as 
slip-stream interfaces, lifting of aircrafts,
tornadoes (refer to \cite{CF, FejerMiles,
miles, SNW} and references therein), so to understand the
stability of contact discontinuities is an important step in studying the multi-dimensional Riemann
problem, the Mach reflection of shocks, the interface problem of
two-phase flow, etc..

In recent years, there are some interesting
works on the stability analysis of contact discontinuities. For
the Euler equations in two-dimensional isentropic unsteady gas
dynamics, in \cite{cou1, cou2} Coulombel and Secchi obtained a rigorous theory on the stability of a supersonic contact discontinuity when the Mach number (the ratio between the relative speed of the
fluid with respect to the discontinuity front over the sonic speed) $M>\sqrt{2}$, which had been investigated already before in \cite{miles} and \cite{AM} by the mode analysis and the nonlinear geometric optics approach respectively. A weakly linear stability result was obtained in \cite{morando-trebeschi} for a
two-dimensional contact discontinuity in nonisentropic compressible
flow. Some related problems on the stability of vortex sheets in two
dimensional steady flow were studied by Chen at al. in
\cite{ChenKukreja, ChenZhangZhu} by using the Glimm scheme.
However, as shown in \cite{serre}, unsteady compressible vortex sheets in three
space dimensions are always
violently unstable, one of main factors is that the tangential velocity fields for the three-dimensional vortex sheets are two-dimensional, and this is the main unstable effect on the vortex sheets. Recently, some works
showed that magnetic fields have stabilization effect on
vortex sheets for two and three dimensional compressible MHD, cf.
refer to \cite{Chen-Wang, Chen-Wang1, tra, tra1, WangYu} and
references therein.

It is an interesting problem to study the
stability of contact discontinuities in three dimensional steady
flows, as it not only plays a crucial role in studying the structural stability of interaction of elementary waves, such as the multi-dimensional
shock reflection-diffraction on an interface, and also shall provide important insight of the really multi-dimensional contact discontinuities, as tangential velocity fields in the three-dimensional steady contact discontinuities are also two-dimensional,  this yields complicated stability phenomena of contact discontinuities. In \cite{WangYu1}, we have obtained the weakly linear stability criteria of contact
discontinuities in three-dimensional compressible isentropic
supersonic steady flows, by computing the Lopatinskii determinant for the linearized problem at a planar contact discontinuity, roughly speaking, it says that for a supersonic contact discontinuity in the three-dimensional steady flow with the velocity fields being non-parallel on both sides of the discontinuity front, it is weakly linearly stable if and only if the velocity fields restricted to a space-like plane should be also supersonic (see \eqref{2.15}). Moreover, we have established  the $L^2$-stability estimates of solutions to the linearized problems of the three-dimensional steady Euler equations \eqref{euler} at a non-planar contact discontinuity by constructing the Kreiss symmterizers through developing the argument from \cite{ kreiss, cou1, WangYu}, and using the para-differential calculus. These estimates exhibit loss of regularity of solutions with respect to the background contact discontinuity and the initial data, it also means that this supersonic contact discontinuity is only weakly stable. 

The main goal of this work is to study the structurally nonlinear stability of the supersonic contact discontinuities in the three-dimensional steady Euler equations \eqref{euler}. As mentioned at above, there is loss of regularity of solutions to the corresponding linearized problem, so we shall adapt the  
Nash-Moser-H\"ormander iteration scheme to study nonlinear problems. From this work, we obtain that a supersonic contact discontinuity satisfying the linearly stable criteria \eqref{2.15} is also  structurally nonlinearly stable at least locally in the propagation direction.

The remainder of this paper is organized as follows. In Section 2,
we formulate the nonlinear problem of a contact discontinuity in
three dimensional compressible isentropic steady flow, and state the structural stability result of
the supersonic  contact discontinuities. To study the nonlinear problem, we establish the tame estimates of solutions to the linearized problem in Section 3. First in \S3.1, we derive the effective linearized problem, and  present the basic $L^2$ stability estimate, then in \S3.2 we derive higher order norm estimates for the lineaized problem. Estimates of tangential derivatives of solutions shall be obtained by differentiating the problems directly. Noting that the discontinuity front is characteristics, from the equations one only can estimate the normal derivatives of non-characteristic components of unknowns in terms of  tangential derivatives of unknowns. To study the characteristic part of unknowns,  inspired from the approach of \cite{cou2}, we introduce a linearized version of vorticity field, and observe each component of vorticity satisfies a transport equation tangential to characteristical boundary, so by combining estimates of vorticity and normal derivatives of non-characteristic unknowns, we conclude the higher order estimates of solutions in \S3.2. In Section 4, we apply the Nash-Moser-H\"ormander iteration to construct the approximate solutions to the nonlinear problem of the supersonic contact discontinuities in the steady Euler equations \eqref{euler}. Finally, the error estimates of the iteration scheme, and the convergence of approximate solutions are given in Section 5, which concludes the structural stability of the supersonic contact discontinuities in the three-dimensional steady Euler flow.

\section{Formulation of Problems and Main Results}

\setcounter{equation}{0}

For the compressible isentropic steady Euler equations \eqref{euler} in three space variables,  assume that the piecewise smooth
function
\begin{equation}\label{cd}
U(x,y,z)=\left\{\begin{array}{l}
U^+(x,y,z), \quad y>\psi(x,z)\\[3mm]
U^-(x,y,z), \quad y<\psi(x,z)\end{array}\right.\end{equation}
with $U=(u,v,w,p)^T$, is a weak
solution of \eqref{euler} in the distribution sense, then it satisfies
the equations
\eqref{euler} classically on both sides of $\Gamma=\{y=\psi(x, z)\}$,
and the Rankine-Hugoniot
jump conditions on the front $\Gamma=\{y=\psi(x, z)\}$:
\begin{equation}\label{rh}
  \psi_x
  \left[\begin{matrix}
    \rho u\\ \rho u^2+p\\ \rho uv\\ \rho uw
  \end{matrix}\right]
-\begin{bmatrix}
 \rho v\\ \rho uv\\ \rho v^2+p\\ \rho vw
\end{bmatrix}
+\psi_z
\begin{bmatrix}
  \rho w\\ \rho uw\\ \rho vw\\ \rho w^2+p
\end{bmatrix}=0,
\end{equation}
with $[\cdot]$ denoting the jump of a related function acrossing the front $\Gamma$. Let $m=\rho(\psi_x u-v+\psi_z w)$ be the mass flux. If $m^{+}=m^{-}=0$ on $\Gamma$, i.e. without any mass transfer flux acrossing the front $\Gamma=\{y=\psi(x, z)\}$,
then $(U^{+}, U^{-}, \Gamma)$ is called a contact discontinuity of \eqref{euler}, in this case, the Rankine-Hugoniot condition \eqref{rh} reads as
\bel{rh1} \psi_x u^{\pm}-v^{\pm}+\psi_z
w^{\pm}=0, \quad p^{+}=p^{-}. \eeq
The first condition given in \eqref{rh1} implies that the normal velocities on both sides of $\Gamma$ vanish, while the tangential velocity fields of $U$ acrossing $\Gamma$ may have jump. As the tangential velocity fields on both sides of $\Gamma$ are two-dimensional, the stability/instability mechanism of the contact discontinuity \eqref{cd} is very challenging, in constrast to the problems of contact discontinuities in the two-dimensional steady or un-steady compressible flows, in which the tangential velocity fields on both sides of front are of one-dimension only.

In this work, we consider the case that the contact
discontinuity \eqref{cd} is supersonic in one direction, say in the $x$-direction, i.e. $u^\pm
>c^\pm$, then as mentioned in Section one, $x$ can be regarded as the time-like. In \cite{WangYu1}, we have studied the linear stability criteria of this supersonic contact discontinuity, and also obtained the $L^2-$estimate of solutions to the problem of the system \eqref{mfeuler} linearized at a background supersonic contact discontinuity.
The aim of this work is to study the structural stability of a supersonic
contact discontinuity. For a given supersonic contact
discontinuity $(U^+, U^-, \Gamma)$ moving from negative $x$ to positive $x$, we are going  to
see whether this contact discontinuity persists in $\{x>0\}$ even for small $x$.
This problem can be formulated as the following one:

\begin{enumerate}
\item[(FBP)]: For a given supersonic contact
discontinuity $(U_0^+, U_0^-)$, $\Gamma_0=\{y=\psi_0(x,z)\}$ of \eqref{mfeuler}-\eqref{rh1} with $u_0^\pm>c_0^\pm$ for $\{x\leq0\}$,
to determine $U^+$, $U^-$ and a free boundary
$\Gamma=\{y=\psi(x,z)\}$ in $\{x>0\}$ satisfying
\begin{equation}\label{FB}
\left\{\begin{array}{l}
A_1(U^\pm)\partial_x
U^\pm+A_2(U^\pm)\partial_y
U^\pm+A_3(U^\pm)\partial_z
U^\pm=0,\qquad\pm(y-\psi(x, z))>0\\[3mm]
\psi_x u^{\pm}-v^{\pm}+\psi_z
w^{\pm}=0, \quad p^{+}=p^{-}, \qquad{\rm on~} \{y=\psi(x, z)\}\\[3mm]
U^\pm|_{x\le 0}=U_0^\pm(x,y,z),
\quad \pm(y-\psi_0(x, z))>0\\[3mm]
\psi|_{x\le 0}=\psi_0(x, z).\end{array}\right.\end{equation}
\end{enumerate}

This is a free boundary problem since the front $\Gamma=\{y=\psi(x,
z)\}$ is also an unknown.
To handle this free boundary,  as
\cite{met2, cou2, Chen-Wang, tra1, WangYu1} we introduce the following transformation from $(x,y,z)$ to $(\tilde{x}, \tilde{y}, \tilde{z})$, 
\bel{psi-y} x=\tilde
x, \quad y=\Psi^\pm(\tilde x, \tilde y, \tilde z), \quad z=\tilde z,
\eeq with $\Psi^\pm(\tilde x,\tilde y,\tilde z)$ satisfying the
 constraints \bel{psi-def}
\begin{cases}
  \Psi^\pm(\tilde x,0,\tilde z)=\psi(\tilde x,\tilde z)\\
  \pm\Psi_{\tilde y}^\pm \ge \kappa_0>0
\end{cases}\eeq
for a positive constant $\kappa_0$, then the domain
$\{\pm(y-\psi(x,z))>0\}$ is changed into $\Omega=\{\tilde y>0\}$ with
the fixed boundary $\{\tilde y=0\}.$

As in \cite{cou2}, inspired by the transport equation of $\psi$ given in \eqref{FB}, the natural candidates of $\Psi^\pm(x,y,z)$ are solutions to the following problem in $\{\tilde y\ge 0\}$:
\begin{equation}\label{2.10}
\begin{cases}
   u^{\pm} \partial_{\tilde x} \Psi^{\pm} -v^{\pm} +w^{\pm} \partial_{\tilde z}\Psi^{\pm}=0,\qquad \tilde x>0\\
\Psi^\pm(0, \tilde y, \tilde z)=\pm \tilde y+\tilde{\Psi}^\pm_0(\tilde y,\tilde z),\end{cases}\end{equation}
with $\tilde{\Psi}^\pm_0(\tilde y, \tilde z)$ being a proper extension of $\psi_0(\tilde z)=\psi_0(0, \tilde z)$ in $\{\tilde y\ge 0\}$, such that
$ \pm\Psi_{\tilde y}^\pm \ge \kappa_0>0$ holds.

Set  $$\tilde{U}_0^\pm(\tilde y,\tilde z)=
U_0^\pm(0,y,z),  \quad \tilde{U}^\pm(\tilde x, \tilde y, \tilde z)= U^\pm (x,y,z).$$
From \eqref{FB}, we know that $\tilde{U}^\pm(\tilde x, \tilde y, \tilde z)$ satisfies the following problem,
\begin{equation}\label{2.11}
\begin{cases}
L(U^\pm, \Psi^\pm)U^\pm=0, \qquad \text{in} ~\{y>0\}\\
{\mathcal B}(U^+, U^-, \psi)=0, \qquad \text{on} ~\{y=0\}\\
U^\pm|_{x=0}=U_0^\pm(y,z), \quad \psi|_{x=0}=\psi_0(z)
\end{cases}\end{equation}
where we have dropped the tildes of notations for simplicity, and 
\begin{equation}\label{212}
\begin{array}{l}
L(U^\pm, \Psi^\pm)U^\pm=A_1(U^{\pm})\partial_xU^{\pm}+\frac{1}{\Psi_y^{\pm}}(A_2(U^{\pm})-\Psi_x^{\pm}A_1(U^{\pm})-\Psi_z^{\pm}A_3(U^{\pm}))\partial_y
  U^{\pm}+A_3(U^{\pm})\partial_zU^{\pm}\end{array},\end{equation}
$${\mathcal B}(U^+, U^-, \psi)=\left(
\begin{array}{c}
\psi_x u^{+}-v^{+}+\psi_z w^{+}\\[2mm]
\psi_x u^{-}-v^{-}+\psi_z w^{-}\\[2mm]
p^{+}-p^{-}\end{array}
\right)$$ with $\psi(x,z)=\Psi^{\pm}(x,0,z)$
and $\Psi^{\pm}(x,y,z)$ being given in \eqref{2.10}.

\vspace{.1in}

To study the nonlinear problem \eqref{2.11}, let us first give a stable background state.
Obviously, the following piecewise constant function
\bel{ubar}
\ol{U}(x,y,z)=\begin{cases}
\ol{U}_r=(\bar{u}_r, 0, \bar{w}_r, \bar{p})^T, \qquad y>0\\[2mm]
\ol{U}_l=(\bar{u}_l, 0, \bar{w}_l, \bar{p})^T, \qquad y<0
\end{cases}
\eeq satisfying
\bel{2.13}
\bar{u}_r>\bar{c}, \quad \bar{u}_l>\bar{c}, ~\quad
~(\bar{u}_r-\bar{u}_l)^2+(\bar{w}_r-\bar{w}_l)^2\neq0,\eeq
 with
$\bar{c}^2=p'(\bar{\rho})$, is a planar contact discontinuity for the compressible steady Euler equations \eqref{euler}. Here $\bar{\rho}$ is the density corrsponding to the pressure $\bar p$ by the relation $p=p(\rho)$. As shown in \cite{WangYu1}, when the tangential velocity fields $(\bar{u}_r, \bar{w}_r)$
and $(\bar{u}_l, \bar{w}_l)$ are parallel, the planar contact discontinuity $(\ol{U}_r, \ol{U}_l)$ is always nonlinearly unstable,  thus in this work we shall only consider the case of $(\bar{u}_r, \bar{w}_r)$
and $(\bar{u}_l, \bar{w}_l)$ being non-parallel to each other. As noted in Remark 2.1 of \cite{WangYu1}, without loss of generality we can assume
\bel{2.14}
\bar w_r \bar w_l<0.\eeq

We impose the following stability conditions obtained in Theorem 3.1 of \cite{WangYu1} on the state $(\ol{U}_r, \ol{U}_l)$ such that the given planar contact discontinuity \eqref{ubar} is weakly stable,
\bel{2.15}
\qquad\qquad\begin{cases}
 \frac{\bar c^2}{\bar u_r^2}+\frac{\bar c^2}
 {\bar u_l^2}<1, \quad \bar w_r^2> \bar c^2, \quad \bar w_l^2> \bar c^2,\\
 \min\limits_{\tilde {\theta} \in (\theta_l, \theta_r)}
 \left(\frac{\bar c^2}{(\bar u_l\sin\tilde{\theta} -\bar w_l\cos\tilde{\theta})^2}
 +\frac{\bar c^2}{(\bar u_r\sin\tilde{\theta} -\bar w_r\cos\tilde{\theta})^2}\right)
  < 1,\\
(\bar u_l \bar w_r-\bar u_r \bar w_l)^2
  \neq 2(\bar c \bar u_l +\bar c \bar u_r)^2+2(\bar c \bar w_l +\bar c
  \bar w_r)^2,
\end{cases}
\eeq
where $\theta_r=\max(\arctan\frac{\bar w_l}{\bar u_l}, \arctan\frac{\bar w_r}{\bar u_r})$ and $\theta_l=\min(\arctan\frac{\bar w_l}{\bar u_l}, \arctan\frac{\bar w_r}{\bar u_r})$.

The main proposal of this work is to prove the following structural stability of vortex sheet $(\ol{U}_r, \ol{U}_l)$.

\begin{theorem} Suppose that the planar contact discontinuity \eqref{ubar} satisfies the conditions \eqref{2.14} and \eqref{2.15}.  Then, for any fixed $s>\frac{27}{2}$, there is a small quantity $\delta>0$ depending on $(\ol{U}_r, \ol{U}_l)$ such that when the initial data $U_0^\pm(y,z)$ and $\psi_0(z)$ given in \eqref{2.11} satisfy
\bel{2.16}
\|U_0^\pm-\overline{U}_{r,l}\|_{H^{s-\frac 12}(\R^2)}+\|\psi_0\|_{H^s(\R)}\le \delta
\eeq
and the compatibility condition of the problem \eqref{2.11} up to order $s-1$, there is $X>0$ such that the problems  \eqref{2.11} and \eqref{2.10}  admit unique solutions
\bel{2.17}
U^\pm\in H^{s-7}([0,X]\times \R_y^+\times \R_z), \qquad
\Psi^\pm \in H^{s-7}([0,X]\times \R_y^+\times \R_z).
\eeq
\end{theorem}


\section{Tame Estimates of Linearized Problems}

\setcounter{equation}{0}

As shown in \cite{WangYu1}, the supersonic contact discontinuity $(\overline{U}_r, \overline{U}_l)$ given by \eqref{ubar} is only weakly linearly stable, as in \cite{cou2, Chen-Wang, tra1} we shall adapt the  Nash-Moser-H\"ormander iteration scheme to study the nonlinear problems \eqref{2.11} and \eqref{2.10}. To do this, in this section, first we derive an effective linearized problem at a non-planar supersonic contact discontinuity, present the $L^2$ stability estimate given in \cite{WangYu1}, then we estimate the solutions of linearized problem in higher order norms.

\subsection{The effective linearized problem and $L^2$-estimate}\label{vclp}

Suppose that a perturbed non-planar contact discontinuity of \eqref{ubar} takes
the form
\bel{vbstate}
U(x,y,z)=\begin{cases}
U_r\xy=\ol{U}_r+V_r\xy,\quad &y>\psi(x,z)\\[1mm]
U_l\xy=\ol{U}_l+V_l\xy ,\quad &y<\psi(x,z)\end{cases}
\eeq
satisfying the Rankine-Hugoniot conditions \eqref{rh1} on $\{y=\psi(x,z)\}$.
To derive the linearized problem of \eqref{mfeuler} and (\ref{rh1}) at the given contact discontinuity solution (\ref{vbstate}), as in \eqref{psi-y}, we take the transformation,
\[ 
x=\tilde{x}, \quad
y=\Psi_{r,l}(\tilde{x}, \tilde{y}, \tilde{z}), \quad
z=\tilde{z}
\]
with $\Psi_{r,l}$ satisfying
\begin{equation}
\begin{cases}
  \Psi_{r,l}(\tilde x,0,\tilde z)=\psi(\tilde x,\tilde z)\\
  \pm\partial_{\tilde y}\Psi_{r.l}\ge \kappa_0>0
\end{cases}\end{equation}
for a positive constant $\kappa_0$.

Set
$$\tilde{U}_{r,l}(\tilde{x}, \tilde{y}, \tilde{z})=
{U}_{r,l}(\tilde{x}, \Psi_{r,l}(\tilde{x}, \tilde{y}, \tilde{z}), \tilde{z}),
\quad
\tilde{V}_{r,l}(\tilde{x}, \tilde{y}, \tilde{z})=
{V}_{r,l}(\tilde{x}, \Psi_{r,l}(\tilde{x}, \tilde{y}, \tilde{z}), \tilde{z})$$
and drop the tildes of notations $\tilde{U}_{r,l}(\tilde{x}, \tilde{y}, \tilde{z})$,
$\tilde{V}_{r,l}(\tilde{x}, \tilde{y}, \tilde{z})$ and $\Psi_{r,l}(\tilde{x}, \tilde{y}, \tilde{z})$
for simplicity in the following calculations.

For a fixed $X>0$, denote by
$\Omega_X=\{\xy~|~0\le x\le  X, y\in \R^+, z\in \R\}.$

For the contact discontinuity \eqref{vbstate}, we impose the following assumptions on the
perturbations:
\bel{3.3}
\begin{cases}
 {V}_r, ~{V}_l,
~\nabla\widetilde{\Psi}_r,
   ~\nabla\widetilde{\Psi}_l \in W^{2, \infty}(\Omega_X),\\
 V_{r}, ~V_l,
  ~\nabla\widetilde{\Psi}_r ~\text{and}
   ~\nabla\widetilde{\Psi}_l ~\text{have compact support in} ~(y,z)\in \R^2_+,\\
  \|(V_r, V_l)\|_{W^{2, \infty}(\Omega_X)}+
   \|(\nabla\widetilde{\Psi}_r, \nabla\widetilde{\Psi}_l)\|_{W^{2,
   \infty}(\Omega_X)} \leq K, \text{for a constant} ~K>0,
\end{cases}
\eeq
where
\bel{vbstate2}
\widetilde{\Psi}_r\xy=\Psi_{r}\xy-y, \quad
\widetilde{\Psi}_l\xy=\Psi_{l}\xy+y . \eeq

Letting $(U^{\pm}, \Phi^{\pm})$ be the small perturbation of the
contact discontinuity $(U_{r,l}\xy,$ $\Psi_{r,l}\xy)$,
from
\eqref{2.11} we get the following linearized equations of
$(U^\pm, \Phi^\pm)$ at $(U_{r,l},\Psi_{r,l})$:
\bel{vlieq} L'(U_{r,l}, \nabla\Psi_{r,l})(U^{\pm},
\Phi^{\pm})=f^{\pm}, \eeq where \bel{vdefL}
\begin{split}
L'(U_{r,l}, \nabla\Psi_{r,l})(U^{\pm}, \Phi^{\pm})=&L(U_{r,l},
\nabla\Psi_{r,l})U^{\pm}
  +C(U_{r,l}, \nabla U_{r,l}, \nabla\Psi_{r,l})U^\pm
\\
  &-\frac{\pd_y\Phi^{\pm}}{(\pd_y\Psi_{r,l})^2}\left( A_2(U_{r,l})-\pd_x\Psi_{r,l}A_1(U_{r,l})
  -\pd_z\Psi_{r,l}A_3(U_{r,l})\right)\pd_yU_{r,l}\\
&-\frac{1}{\pd_y\Psi_{r,l}}\left(\pd_x\Phi^{\pm}A_1(U_{r,l})+\pd_z\Phi^{\pm}A_3(U_{r,l})\right)\pd_yU_{r,l}
\end{split}
\eeq
in which
\bel{defl} L(U_{r,l}, \nabla\Psi_{r,l})U
=A_1(U_{r,l})\partial_xU+A_b(U_{r,l}, \nabla\Psi_{r,l})\partial_y
  U+A_3(U_{r,l})\partial_zU,\eeq
with
\[A_b(U, \nabla\Psi)=\frac{1}{\pd_y\Psi}(A_2(U)-\pd_x\Psi A_1(U)-
\pd_z\Psi A_3(U))
\]
and
\bel{defc}
\begin{split}
C(U_{r,l}, \nabla U_{r,l}, \nabla\Psi_{r,l})U=&(\nabla A_1(U_{r,l})U)\pd_xU_{r,l} +
(\nabla A_3(U_{r,l})U)\pd_zU_{r,l}\\
&\hspace{-.8in}+\frac{1}{\pd_y \Psi_{r,l}}
[\nabla A_2(U_{r,l})U-(\nabla A_1(U_{r,l})U)\pd_x\Psi_{r,l}-(\nabla A_3(U_{r,l})U)\pd_z\Psi_{r,l}]\pd_yU_{r,l}\,.
\end{split}
\eeq

When $\Psi_{r,l}(x,y,z)$ satisfy the eikonal equations
\bel{eikonal}
   u_{r,l} \partial_x \Psi_{r,l} -v_{r,l} +w_{r,l} \partial_{z}\Psi_{r,l}=0
\eeq
in $\{y\ge 0\}$,  we know that the boundary matrix
\[
A_b(U_{r,l}, \nabla\Psi_{r,l}) =\frac{1}{\pd_y\Psi_{r,l}}\begin{pmatrix}
  0&0&0&-\pd_x\Psi_{r,l}\\
  0&0&0&1\\
  0&0&0&-\pd_z\Psi_{r,l}\\
  -\pd_z\Psi_{r,l}&1&-\pd_z\Psi_{r,l}&0
\end{pmatrix},
\]
has a constant rank in the domain $\Omega_X$.

As the first order derivatives of $U^\pm$ and $\Phi^\pm$ are coupled
together in the equations \eqref{vlieq}, to deal with this
problem, as in \cite{ali}, by introducing the ``good unknowns"
\bel{vgu} U_{+}=U^{+}-\frac{\Phi^{+}}{\pd_y\Psi_r}\pd_y U_r, \quad
U_{-}=U^{-}-\frac{\Phi^{-}}{\pd_y\Psi_l}\pd_y U_l, \eeq we obtain the equations for
$U_{\pm}$,
 \bel{lieq2} L(U_{r,l},
\nabla\Psi_{r,l})U_{\pm}+\frac{\Phi^{\pm}}{\pd_y\Psi_{r,l}} \pd_y[L(U_{r,l},
\nabla\Psi_{r,l})U_{r,l}]+C(U_{r,l}, \nabla U_{r,l}, \nabla\Psi_{r,l})U_{\pm}=f^{\pm}. \eeq
in which $\Phi^{\pm}$ is appeared only in the zero-th order terms.  By shifting these zero-th order terms into the source terms $f^\pm$, we obtain that $U_\pm$ satisfy  the
following effective linear equations, 
\bel{lieqr} L'_{e}(U_{r,l}, \nabla\Psi_{r,l})U_{\pm}=L(U_{r,l},
\nabla\Psi_{r,l})U_{\pm}+C(U_{r,l}, \nabla U_{r,l}, \nabla\Psi_{r,l})U_{\pm}=f^{\pm}. \eeq

In terms of the good unknowns $U=(U_{+}, U_{-})^T$, the linearization of
the boundary conditions given in \eqref{2.11} is given by
\bel{lbc} B'_e(U,
\phi)=\underline{b}(x,z)\nabla\phi+\ol b(x,z)\phi
+\underline{M}(x,z)U|_{y=0}=g, \qquad
{\rm on}\quad y=0
\eeq with
$\phi=\Phi^{+}|_{y=0}=\Phi^{-}|_{y=0}$ and \bel{coebc}
\underline{b}(x,z)=\begin{pmatrix}
  u_r&w_r\\ u_l&w_l\\0&0
\end{pmatrix}_{|{y=0}},
\qquad \ol b=\underline{M}(x,z)
\begin{pmatrix}
  \frac{\pd_yU_r}{\pd_y\Psi_r}\\[3mm] \frac{\pd_yU_l}{\pd_y\Psi_l}
\end{pmatrix}_{|{y=0}}\,,
\eeq
\bel{coebc-1}
 \underline{M}(x,z)=\begin{pmatrix}
  \psi_x&-1&\psi_z&0&0&0&0&0\\
  0&0&0&0&\psi_x&-1&\psi_z&0\\
  0&0&0&1&0&0&0&-1
\end{pmatrix}.
\eeq

Therefore, the effective linear problem of $U$ is formulated as
\bel{lpb}
\begin{cases}
L'_eU_{+}=f^{+}, \quad L'_eU_{-}=f^{-}, \quad \text{in} \,\, \Omega_{X}\\
B_e'(U, \phi)=g, \quad \text{on} \,\, \{y=0\}
\end{cases}
\eeq where $U, \Phi, f^+, f^-$ and $g$ vanish in $\{x\leq0\}.$

This problem has been studied by authors in \cite{WangYu1} throughly. To recall the $L^2$ stability estimate given in \cite{WangYu1}, we
first introduce the weighted Sobolev space $H^s_{\gamma}$ for
$\gamma \geq 1$, $s\in\R$ as
\[
H^s_{\gamma}(\Omega_X)=\{u\in \mathcal{D}'(\Omega_X) \left. \right|
e^{-\gamma x}u\in H^s(\Omega_X)\},\] 
with the norm
\[\|u\|_{H_\gamma^s(\Omega_X)}= \|e^{-\gamma x} u\|_{H^s(\Omega_X)}.\]
The space $L^2(\R^{+}_y;~H^s_{\gamma}(\omega_X))$ defined in the domain
$\Omega_X$, is endowed with the norm
\[
\|u\|_{L_y^2(H_\gamma^s)}=
\left(\int_0^{+\infty}\|u(\cdot,y)\|_{H_\gamma^s(\omega_X)}^2 \dd
y\right)^{\frac{1}{2}},\]
where $\omega_X=\Omega_X\cap\{y=const.\}=\{(x,z)|~0\le x\le X, z\in\R\}$
and the space $H^k(\R^{+}_y;~H^s_{\gamma}(\omega_X))$ can be defined similarly.

In Theorem 4.1 of \cite{WangYu1}, by using the paradifferential calculus and constructing the Kreiss symmetrisers we have obtained the following energy estimate for the problem \eqref{lpb}:

\begin{theorem} (\cite{WangYu1})\label{L2estimate}
Let $(\ol U_r, \ol U_l)$ defined in \eqref{ubar} be the planar
supersonic contact discontinuity satisfying the stability assumptions \eqref{2.14} and
\eqref{2.15},  and its perturbed non-planar contact discontinuity
\eqref{vbstate} satisfies the condition \eqref{3.3} for a
constant $K>0$. Then for the linear problem \eqref{lpb},
 there exist constants $K_0>0$ depending on the contact discontinuity
  $(\ol U_r, \ol U_l)$, and  $\gamma_0\geq1$, $C_0>0$ depending on $K_0$
such that for all $K\leq K_0$, $\gamma \geq \gamma_0$ and $(U,
\phi)\in H_\gamma^2(\Omega_X)\times H_\gamma^2(\omega_X)$, one has
 \bel{vestimate}
 \gamma\|U\|_{L_\gamma^2(\Omega_X)}^2 +\|\PP U|_{y=0}\|_{L_\gamma^2(\omega_X)}^2
 +\|\phi\|_{H_\gamma^1(\omega_X)}^2 \leq C_0\left(
 \frac{1}{\gamma^3}\|f\|_{L_y^2(H_\gamma^1)}^2 +\frac{1}{\gamma^2}
 \|g\|_{H_\gamma^1(\omega_X)}^2
 \right),
 \eeq
where $U=(U_+, U_-)^T$, $f=(f^+, f^-)^T$ and
 \bel{vdefpp}
 \PP U_{\pm}=\begin{pmatrix}
    \psi_xU_{\pm,1}-U_{\pm,2}+\psi_zU_{\pm,3}\\
    U_{\pm,4}
  \end{pmatrix}\,.\eeq
\end{theorem}

\subsection{Higher order estimates of solutions to the linearized problem}

 In this section, we are going to derive the higher order
estimates of the solution $(U, \phi)$ to the linearized boundary
value problem \eqref{lpb}, this is the key step for studying the nonlinear problem \eqref{2.11} by using the
Nash-Moser-H\"ormander iteration scheme in next section.

Assuming that for a fixed $s\in \N$, the perturbation $(V_{r,l}, \nabla \widetilde \Psi_{r,l})$
belongs to $H_\gamma^{s+2}(\Omega_X)\cap H_\gamma^5(\Omega_X)$, the main result of this section is the following one:

\begin{theorem}\label{hsestimate}
Let $s\in \N$ and $X>0$. Assume that the non-planar contact discontinuity
given in \eqref{vbstate} satisfies \eqref{2.15}, \eqref{3.3} and
$(V_{r,l}, \nabla \widetilde \Psi_{r,l})\in
H_\gamma^{s+2}(\Omega_X)\cap H_\gamma^{5}(\Omega_X)$ with
\bel{asmp}
\|\nabla
\widetilde \Psi\|_{H_\gamma^{5}(\Omega_X)}+\|
V\|_{H_\gamma^{5}(\Omega_X)}\leq K,\eeq
with $\widetilde \Psi=(\widetilde \Psi_{r}, \widetilde \Psi_{l})^T$ and $V=(V_r, V_l)^T$.
Then, for the problem \eqref{lpb}, there exist constants
$K_s>0$ and $\gamma_s\ge 1$
depending on $s$,  such that for all $K\leq K_s$,
$\gamma\geq \gamma_s,$ and $(U, \phi)\in
(H_\gamma^{s+2}(\Omega_X)\times H_\gamma^{s+2}(\omega_X))\cap (H_\gamma^{5}(\Omega_X)\times H_\gamma^{5}(\omega_X))$, one has
\bel{test}
\begin{split}
&\sqrt{\gamma}\|U\|_{H_\gamma^s(\Omega_X)} +\|\PP
U|_{y=0}\|_{H_\gamma^s(\omega_X)}
+\|\phi\|_{H_\gamma^{s+1}(\omega_X)} \leq
C(K)\left\{\frac{1}{\gamma^{3/2}}\|f\|_{H_\gamma^{s+1}(\Omega_X)}
+\frac{1}{\gamma}\|g\|_{H_\gamma^{s+1}(\omega_X)}\right.\\
&
\hspace{.3in}+\left.\left(\frac{1}{\gamma^{3/2}}\|(V, \nabla \widetilde
\Psi)\|_{H_\gamma^{s+2}(\Omega_X)}+\frac{1}{\gamma}\|\pd_y
V_{|y=0}\|_{H_\gamma^{s}(\omega_X)}\right)(\|f\|_{H^4(\Omega_X)}
  +\|g\|_{H^4(\omega_X)})
\right\},
\end{split}
\eeq
where $C(K)$ is a positive constant depending on $K$.

\end{theorem}

In the proof of this theorem, we shall always use $C(K)$ to denote a general positive constant depending on $K$, which may change from line to line, and shall frequently use the
following elementary inequalities, which can be found in textbooks, e.g. \cite{evans}:

\begin{enumerate}
\item[(1)] The Gagliardo-Nirenberg inequality,
\[
\|\pd^\alpha u\|_{L_\gamma^p(\Omega_X)}\leq
C\|u\|_{L^\infty(\Omega_X)}^{1-2/p}\|u\|_{H_\gamma^s(\Omega_X)}^{2/p}
\]
with $\frac{2}{p}=\frac{|\alpha|}{s}$, for all $u\in
H_\gamma^s(\Omega_X) \bigcap L^\infty(\Omega_X).$

\item[(2)] Let $F$ be a $C^\infty$ function defined on $\R^n$
and satisfy $F(0)=0$. Then, for all $u\in
H^s_\gamma(\Omega_X)\cap L^\infty(\Omega_X),$ one has \bel{ff}
\|F(u)\|_{H_\gamma^s(\Omega_X)}\leq
C(\|u\|_{L^\infty(\Omega_X)})\|u\|_{H_\gamma^S(\Omega_X)}.\eeq
\end{enumerate}

To prove the higher order estimate \eqref{test}, first we shall study tangential derivatives by using the $L^2-$estimate given in Theorem \ref{L2estimate}, then from the equations  \eqref{lpb} we estimate the normal derivatives of the non-characteristic components of unknowns, finally to estimate the normal derivatives of the characteristic components we study the problems of vorticity fields derived from the problem \eqref{lpb}.
These estimates will be given in the following subsections.

\subsubsection{Estimate of tangential derivatives}

We introduce the following
transformations in the problem \eqref{lpb} to diagonalize the boundary matrices $A_b(U_{r,l},
\nabla\Psi_{r,l})$ of the effective linear equations \eqref{lieqr}, 
\bel{utow}
 W^+=T(\nabla \Psi_r)U_+, \quad
W^-=T(\nabla \Psi_l)U_-,
\eeq
where
\[T(\nabla\Psi)=\begin{pmatrix}
  1& 0& -\pd_x\Psi& -\pd_x\Psi\\
  \pd_x\Psi& \pd_z\Psi& 1& 1\\
  0& 1& -\pd_z\Psi& -\pd_z\Psi\\
  0& 0& \lc\pd\Psi\rc& -\lc\pd\Psi\rc
\end{pmatrix}^{-1}
\]
with $\lc\pd\Psi\rc=\sqrt{1+(\pd_x\Psi)^2+(\pd_z\Psi)^2}$, and
multiply \[ A_0(\nabla\Psi_{r,l})=\text{diag}(1, 1,
\frac{\pd_y\Psi_{r,l}}{\lc\pd\Psi_{r,l}\rc},
-\frac{\pd_y\Psi_{r,l}}{\lc\pd\Psi_{r,l}\rc})
\] from the left hand side of the equations of $W^+, W^-$ respectively. It's easy to get that
$W=(W^+, W^-)^T$ satisfies the following problem,
\bel{pb}
\begin{cases}
A_1^r\pd_xW^{+}+\Lambda_2\pd_yW^{+}+A_3^r\pd_zW^{+}+A_0^rC^rW^{+}=F^{+},
\quad {\rm in} \,\,\Omega_X\\
A_1^l\pd_xW^{-}+\Lambda_2\pd_yW^{-}+A_3^l\pd_zW^{-}+A_0^lC^lW^{-}=F^{-},
\quad {\rm in} \,\,\Omega_X\\
\underline{b}\nabla\phi +\ol b\phi +\underline{M}\begin{pmatrix}
  T^{-1}_r&0\\ 0&T_l^{-1}
\end{pmatrix}W|_{y=0} =g,\\
 (W^+,W^-,\phi)|_{x\le 0}=0,
\end{cases}
\eeq where $\Lambda_2=\text{diag}(0, 0, 1,
  1)$, $F^{\pm}=A_0^{r,l}T_{r,l}f^{\pm}$ and $g$ vanish for $x\le 0$,
\begin{equation}\label{vdefac}
\begin{split}
  &A_1^{r,l}=A_0^{r,l}T_{r,l}A_1T_{r,l}^{-1}, \quad   A_3^{r,l}=A_0^{r,l}T_{r,l}A_3
  T^{-1}_{r,l}, \\
  &C^{r,l}=T_{r,l}A_1\pd_xT_{r,l}^{-1}+T_{r,l}A_3\pd_zT_{r,l}^{-1}+T_{r,l}A_b\pd_y
  T_{r,l}^{-1}+T_{r,l}C T_{r,l}^{-1}
\end{split}
\end{equation}
with notations $T_{r,l}=T(\nabla\Psi_{r,l}),~
A_0^{r,l}=A_0(\nabla\Psi_{r,l})$.

 From \eqref{vdefac}, we know that $A_0^{r,l}$, $A_{1}^{r,l}$ and $A_{3}^{r,l}$ are
$C^\infty$ functions of $(U_{r,l}, \nabla\Psi_{r,l})$, and
$C^{r,l}$ are $C^\infty$ functions of
$(U_{r,l}, \nabla U_{r,l}, \nabla\Psi_{r,l}, \nabla\partial_x\Psi_{r,l}, \nabla\partial_z\Psi_{r,l})$.

\begin{lemma}\label{let}
  For any $s\in \N$, under the assumptions of Theorem \ref{hsestimate}, there exists a constant $C(K)>0$ such that the following estimate holds for the solution of \eqref{pb},
  \bel{testimate}
  \begin{split}
  &\sqrt{\gamma}\|W\|_{L_y^2(H_\gamma^s)} +\|W^{nc}_{|y=0}\|_{H_\gamma^s(\omega_X)}
  +\|\phi\|_{H_\gamma^{s+1}(\omega_X)} \\
  &\leq C(K)\left\{\frac{1}{\gamma^{3/2}}\|F\|_{L_y^2(H_\gamma^{s+1})}
  +\frac{1}{\gamma}\|g\|_{H_\gamma^{s+1}(\omega_X)}
  +\frac{1}{\gamma^{3/2}}\|W\|_{L^{\infty}(\Omega_X)}
\left(  \|(V, \nabla \widetilde \Psi)\|_{L_y^2(H_\gamma^{s+2})}+\|\partial_y V\|_{L_y^2(H_\gamma^{s+1})}\right)\right.\\
  &\left.\quad+\frac{1}{\gamma}\left(\| W^{nc}_{|y=0}\|_{L^\infty(\omega_X)}
  +\|\phi\|_{W^{1,\infty}(\omega_X)}\right)\|(V, \pd_y V,
   \nabla \widetilde \Psi)_{|y=0}\|_{H_\gamma^{s}(\omega_X)}\right\},
  \end{split}
  \eeq
where $W^{nc}=(W_3^+, W_4^+, W_3^-, W_4^-)^T$.
\end{lemma}

\begin{proof}
There are three steps to obtain the estimate \eqref{testimate}.

(1) Define $l$-th order tangential operator
$\pd_T^\alpha=\pd_x^{\alpha_1}\pd_z^{\alpha_2}$ with
$|\alpha|=\alpha_1+\alpha_2=l$, for some $l\le s$. From \eqref{pb}, we know that
\[
A_1^r\pd_x\pd_T^\alpha W^++\Lambda_2\pd_y\pd_T^\alpha
W^++A_3^r\pd_z\pd_T^\alpha W^++A_0^rC^r\pd_T^\alpha W^+
+[\pd_T^\alpha, ~A_1^r\pd_x+A_3^r\pd_z+A_0^rC^r]W^+=\pd_T^\alpha
F^+,
\]
where $[\cdot,\cdot]$ denotes the commutator. We introduce the notation
$a^{(l)}$ as an element of the set $\{\pd^\alpha_T a~:~|
\alpha|=l\}$ for any function $a$ belonging to $W^{s,
\infty}(\Omega_X)$ or $W^{s, \infty}(\omega_X)$, and rewrite the
above equations as \bel{eqwl1}
\begin{split}
&A_1^r\pd_xW_+^{(l)}+\Lambda_2\pd_yW_+^{(l)}+A_3^r\pd_zW_+^{(l)}+A_0^rC^rW_+^{(l)}
+\sum\limits_{|\beta|=1} C_{\alpha, \beta} (\pd_T^\beta
A_1^r\pd_T^{\alpha-\beta}
\pd_xW^++\pd_T^\beta A_3^r\pd_T^{\alpha-\beta}\pd_zW^+)\\
&\hspace{.2in}=\pd_T^\alpha F^+ -\sum\limits_{|\beta|\geq 2, \beta\leq \alpha}
C_{\alpha, \beta} (\pd_T^\beta
A_1^r\pd_T^{\alpha-\beta}\pd_xW^++\pd_T^\beta
A_3^r\pd_T^{\alpha-\beta}\pd_zW^+) -[\partial_T^\alpha, A_0^rC^r]W^+\,,
\end{split}
\eeq where $C_{\alpha, \beta}$ are constants depending on $\alpha,
\beta$. The equations of $W_-^{(l)}$ are similar to \eqref{eqwl1}.
The corresponding boundary conditions of
$W^{(l)}=(W_+^{(l)},W_-^{(l)})^T$ on $\{y=0\}$ are \bel{bcwl1} \underline{b} \nabla
\pd_T^\alpha\phi +\overline{b}\pd_T^\alpha\phi +M \pd_T^\alpha
W^{nc} =\pd_T^\alpha g
-[\partial_T^\alpha, \underline{b} \nabla
 +\overline{b}]\phi -[\pd_T^\alpha, M]
W^{nc}\,.
\eeq
For simplicity of notations,
we rewrite the above problem of $W^{(l)}=(W^{(l)}_+, W^{(l)}_-)^T$ as \bel{eqwl}
\begin{cases}
   \ha_1 \pd_x W^{(l)}+\Lambda_4\pd_yW^{(l)} +\ha_3\pd_zW^{(l)} +\hc W^{(l)}
   =\hf^{(l)}, \quad \text{in} ~\Omega_X\\
  \underline{b} \nabla \phi^{(l)} +\overline{b}\phi^{(l)} +M W^{(l), nc}
  =g^{(l)}, \quad \text{on}~ \{y=0\}
\end{cases}
\eeq where $\Lambda_4=\text{diag}(0,0,1,1,0,0,1,1),$
\[
\ha_{k}=\begin{pmatrix}
  A_{k}^r &\\ &A_{k}^l
\end{pmatrix} \quad (k=1,3), 
~ \hc=\begin{pmatrix}
  \widetilde C^r &\\ &\widetilde C^l
\end{pmatrix},
~ \hf^{(l)}=(F_+^{(l)}, F_-^{(l)})^T,
\]
 with $\widetilde C^r=A_0^rC^r+\sum\limits_{|\beta|=1}
C_{\alpha, \beta} (\pd_T^\beta A_1^r+\pd_T^\beta A_3^r),$
$F_+^{(l)},g^{(l)}$ denoting the right hand sides of \eqref{eqwl1}
and \eqref{bcwl1} respectively, and $F_-^{(l)}$ defined similar
to $F_+^{(l)}$.
$M$ in the boundary conditions is the nonzero submatrix of $\underline{M}\begin{pmatrix}
  T^{-1}_r&0\\ 0&T_l^{-1}
\end{pmatrix}$,

Noting that the coefficients in the equations given in \eqref{eqwl}
all belong to $\w2(\Omega_X)$ except for $\hc\in W^{1,
\infty}(\Omega_X),$ we can apply Theorem \ref{L2estimate} in the problem \eqref{eqwl} to get
\bel{esh} \gamma\|W^{(l)}\|_{L_\gamma^2(\Omega_X)}^2 +\|
W^{(l),nc}_{|y=0}\|_{L_\gamma^2(\omega_X)}^2
 +\|\phi^{(l)}\|_{H_\gamma^1(\omega_X)}^2 \leq C_0\left(
 \frac{1}{\gamma^3}\|\hf^{(l)}\|_{L_y^2(H_\gamma^1)}^2 +\frac{1}{\gamma^2}\|g^{(l)}\|_{H_\gamma^1(\omega_X)}^2
 \right).
\eeq

(2) To obtain \eqref{testimate}, it remains to estimate the source
terms $\hf^{(l)}$ and $g^{(l)}.$ From \eqref{eqwl1},  we get
\[\begin{split} F_+^{(l)}=&
\pd_T^\alpha F^+ -\sum\limits_{|\beta|\geq 2, \beta\leq \alpha}
C_{\alpha, \beta} (\pd_T^\beta
A_1^r\pd_T^{\alpha-\beta}\pd_xW^++\pd_T^\beta
A_3^r\pd_T^{\alpha-\beta}\pd_zW^+) -[\partial_T^\alpha, A_0^rC^r]W^+.
\end{split}\]
Obviously, we have \bel{fl} \|\pd_T^\alpha
F^+\|_{L_y^2(H_\gamma^1)}\leq \|F^+\|_{L_y^2(H_\gamma^{l+1})}, \quad
|\alpha|= l. \eeq
Applying H\"{o}lder's and Gagliardo-Nirenberg's
inequalities for $\beta\le \alpha$ with $|\beta|\ge 2$, $|\alpha|=l$,
one has
\[
\|\pd_T^\beta A_1^r\pd_T^{\alpha-\beta}\pd_x
W^+\|_{L^2_\gamma(\Omega_X)}\leq C(K)\left(\|W^+\|_{L_y^2(H_\gamma^l)}
+\|(V_r, \nabla \widetilde
\Psi_r)\|_{L_y^2(H_\gamma^{l+1})}\|W^+\|_{L^{\infty}(\Omega_X)}\right),
\]
and
\[
\|\pd_T(\pd_T^\beta A_1^r\pd_T^{\alpha-\beta}\pd_x
W^+)\|_{L^2_\gamma(\Omega_X)}\leq
C(K)\left(\|W^+\|_{L_y^2(H_\gamma^l)} +\|(V_r, \nabla
\widetilde
\Psi_r)\|_{L_y^2(H_\gamma^{l+2})}\|W^+\|_{L^\infty(\Omega_X)}\right).
\]
Thus, by using the following
equivalent relation
\[
\|F\|_{L_y^2(H_\gamma^1)}\simeq\gamma\|F\|_{L^2_\gamma}+\|\pd_T
F\|_{L^2_\gamma},
\]
we obtain \bel{a1} \|\pd_T^\beta A_1^r\pd_T^{\alpha-\beta}\pd_x
W^+\|_{L_y^2(H^1_\gamma)}\leq
C(K)\left(\gamma\|W^+\|_{L_y^2(H_\gamma^l)} +\|(V_r, \nabla
\widetilde
\Psi_r)\|_{L_y^2(H_\gamma^{l+2})}\|W^+\|_{L^\infty(\Omega_X)}\right). \eeq

Similarly, we can deduce that
\bel{a3}
\|\pd_T^\beta A_3^r\pd_T^{\alpha-\beta}\pd_z W^+\|_{L_y^2(H^1_\gamma)}\leq C(K)
  \left(\gamma\|W^+\|_{L_y^2(H_\gamma^l)}
+\|(V_r, \nabla \widetilde
\Psi_r)\|_{L_y^2(H_\gamma^{l+2})}\|W^+\|_
{L^\infty(\Omega_X)}\right),\eeq
for $\beta\le \alpha$ with $|\beta|\ge 2, |\alpha|=l$, and
\bel{a3-1}
\|\pd_T^\beta (A_0^rC^r)\pd_T^{\alpha-\beta}
W^+\|_{L_y^2(H^1_\gamma)}\leq
C(K)\left(\gamma\|W^+\|_{L^2(H_\gamma^{l-1})} +\|(V_r, \partial_yV_r, \nabla
\widetilde
\Psi_r)\|_{L_y^2(H_\gamma^{l+1})}\|W^+\|_{L^\infty(\Omega_X)}\right)
\eeq
as $\beta\le \alpha$ with $|\beta|\ge 1$, by noting that $A_0^rC^r$ is a $C^\infty$ function of $(U_r,
 \nabla U_r, \nabla \Psi_r, \nabla\partial_T\Psi_r)$ which vanishes at the
origin while $A_{3}^r$ are $C^\infty$ of $(U_r, \nabla \Psi_r)$.

Adding \eqref{fl}, \eqref{a1} and \eqref{a3}-\eqref{a3-1}, we have
\begin{equation}\label{4.33}
\begin{array}{ll}
\|F_+^{(l)}\|_{L_y^2(H_\gamma^1)} \leq
& C(K)\left\{
\|F^+\|_{L_y^2(H_\gamma^{l+1})}+ \gamma\|W^+\|_{L_y^2(H_\gamma^l)}\right.\\[2mm]
&\left.+\left(\|(V_r, \nabla \widetilde
\Psi_r)\|_{L_y^2(H_\gamma^{l+2})}+\|\partial_y V_r\|_{L_y^2(H_\gamma^{l+1})}\right)
\|W^+\|_{L^\infty(\Omega_X)}\right\}.
\end{array}
\end{equation}
One can have a similar estimate for $F^{(l)}_-$,  and conclude the following estimate for the source term of the equation given in \eqref{eqwl},
\bel{hfl}
\begin{array}{l}
\|\hf^{(l)}\|_{L_y^2(H_\gamma^1)} \leq
C(K)\left\{
\|F\|_{L_y^2(H_\gamma^{l+1})}+ \gamma\|W\|_{L_y^2(H_\gamma^l)}
+(\|(V, \nabla \widetilde
\Psi)\|_{L_y^2(H_\gamma^{l+2})}+\|\partial_y V\|_{L_y^2(H_\gamma^{l+1})})
\|W\|_{L^\infty(\Omega_X)}\right\},
\end{array}\eeq
where $W=(W^+, W^-)^T$, $F=(F^+, F^-)^T$, $V=(V_r,
V_l)^T$ and $\widetilde \Psi=(\widetilde \Psi_r, \widetilde
\Psi_l)^T$.

The estimate of the term $g^{(l)}$ can be studied similarly. From the right hand side of \eqref{bcwl1}, we get
\bel{hg}
\begin{split}
\|g^{(l)}\|_{H_\gamma^1(\omega_X)} \leq &C(K)\left\{
\|g\|_{H_\gamma^{l+1}(\omega_X)}+
\|\phi\|_{H_\gamma^{l+1}(\omega_X)}+\|W^{nc}_{|y=0}\|_{H_\gamma^{l-1}(\omega_X)}\right.\\
&\left.+\|(V, \pd_y V, \nabla \widetilde
\Psi)_{|y=0}\|_{H_\gamma^{l}(\omega_X)}\|\phi\|_{\ww(\omega_X)}+
\|\nabla \widetilde
\Psi_{|y=0}\|_{H_\gamma^{l}(\omega_X)}\|W^{nc}_{|y=0}\|_{L^\infty(\omega_X)}\right\}.\\
\end{split}
\eeq

(3) Plugging \eqref{hfl} and \eqref{hg} into the right hand side of
\eqref{esh}, one has \bel{tesimate1}
\begin{split}
&\sqrt{\gamma}\|W^{(l)}\|_{L_\gamma^2(\Omega_X)} +\|
W^{(l),nc}_{|y=0}\|_{L_\gamma^2(\omega_X)}
 +\|\phi^{(l)}\|_{H_\gamma^1(\omega_X)} \\
\leq& C(K)\left\{\gamma^{-3/2}\left(\|F\|_{L_y^2(H_\gamma^{l+1})}
+\gamma\|W\|_{L_y^2(H_\gamma^l)}+(\|(V, \nabla \widetilde
\Psi)\|
_{L_y^2(H_\gamma^{l+2})}+\|\partial_y V\|_{L_y^2(H_\gamma^{l+1})})
\|W\|_{L^\infty(\Omega_X)}\right)\right.\\
&+\gamma^{-1}\left(\|g\|_{H_\gamma^{l+1}(\omega_X)}
+\|\phi\|_{H_\gamma^{l+1}(\omega_X)}+\|W^{nc}_{|y=0}\|_{H_\gamma^l(\omega_X)}\right.\\
&\left.\left.+\|(V, \pd_y V, \nabla
\widetilde
\Psi)_{|y=0}\|_{H_\gamma^{l}(\omega_X)}\|\phi\|_{\ww(\omega_X)}+ \|\nabla
\widetilde
\Psi_{|y=0}\|_{H_\gamma^{l}(\omega_X)}\|
W^{nc}_{|y=0}\|_{L^\infty(\omega_X)}\right) \right\}.
\end{split}
\eeq
Then, multiplying $\gamma^{s-l}$ on \eqref{tesimate1} and
taking the summation for $l$ from 0 to $s$, we obtain the estimate
\eqref{testimate} after absorbing the following term
\[
\gamma^{-1/2}\|W\|_{L_y^2(H_\gamma^s)}+\gamma^{-1}(\|W^{nc}_{|y=0}\|_{H_\gamma^s(\omega_X)}
  +\|\phi\|_{H_\gamma^{s+1}(\omega_X)})
\]
by the left hand side of \eqref{testimate}.
\end{proof}

\subsubsection{Estimate of ``vorticities"}
 Since the boundary $\{y=0\}$ in the problem \eqref{pb} is
characteristic for the equations, we can not control the
normal derivatives $\pd_y W^\pm_1$ and  $\pd_y W^\pm_2$ directly from the equations. Here we employ an idea inspired by the approach given in \cite{cou2} to study
``vorticities" from the original equations given in \eqref{lpb} of
$(U_+, U_-)$, these vorticities are represented by $\pd_y W^\pm_1$, $\pd_y W^\pm_2$ and
the tangential derivatives $\pd_TW.$

Obviously, the first three equations of $U_+=(u_+, v_+, w_+, p_+)^T$
given in \eqref{lieqr} can be formulated as, 
\bel{vu+}
\begin{cases}
  \rho_r (u_r \pd_x  +w_r\pd_z) u_+ +(\pd_x-\frac{\pd_x\Psi_r}{\pd_y\Psi_r}\pd_y) p_+=\left(f^+-C(U_r, \nabla U_r, \nabla \Psi_r)U_+\right)_1,\\
  \rho_r (u_r \pd_x  +w_r\pd_z) v_+ +\frac{1}{\pd_y\Psi_r}\pd_y p_+=\left(f^+-C(U_r, \nabla U_r, \nabla \Psi_r)U_+\right)_2,\\
  \rho_r (u_r \pd_x  +w_r\pd_z) w_+ +(\pd_z-\frac{\pd_z\Psi_r}{\pd_y\Psi_r}\pd_y) p_+=\left(f^+-C(U_r, \nabla U_r, \nabla \Psi_r)U_+\right)_3,
\end{cases}
\eeq with $(\cdot)_i$ denoting the $i$-th component of a
vector. If we introduce ``vorticities'' defined by
\bel{defxi} \xi_+=(\pd_x -\frac{\pd_x\Psi_r}{\pd_y\Psi_r}\pd_y)v_+
-\frac{1}{\pd_y\Psi_r}\pd_y u_+, \quad
\zeta_+=(\pd_z-\frac{\pd_z\Psi_r}{\pd_y\Psi_r}\pd_y)v_+-\frac{1}{\pd_y\Psi_r}
\pd_y w_+, \eeq then, from \eqref{vu+} we know that $(\xi_+,
\zeta_+)^T$ satisfy the following transport equations: \bel{eqxi}
\begin{cases}
\rho_r (u_r \pd_x +w_r\pd_z) \xi_+
=(\pd_x-\frac{\pd_x\Psi_r}{\pd_y\Psi_r}\pd_y)
\tilde f_2^+ -\frac{1}{\pd_y\Psi_r}\pd_y\tilde f_1^++R_1\cdot\nabla U_+\,,\\
\rho_r (u_r \pd_x +w_r\pd_z) \zeta_+
=(\pd_z-\frac{\pd_z\Psi_r}{\pd_y\Psi_r}\pd_y)
\tilde f_2^+ -\frac{1}{\pd_y\Psi_r}\pd_y\tilde f_3^++R_2\cdot\nabla U_+\,,\\
(\xi_+,\zeta_+)|_{x\le 0}=0,
\end{cases}
\eeq where \bel{deftf} \tilde f_i^+=(f^+-C(U_r, \nabla U_r, \nabla
\Psi_r)U_+)_i, \quad i=1,2,3 \eeq and $R_1, R_2$ are $C^\infty$
vector valued functions depending on $(V_r, \nabla
V_r, \nabla \widetilde \Psi_r, \nabla^2 \widetilde
\Psi_r)$ and vanish at the origin.

For the problem \eqref{eqxi}, we have the estimates of $\xi_+$ and $\zeta_+$ as follows,

\begin{lemma} Let $s>1$, there exist constants $C(K)>0$ and $\gamma_s\ge1$ such
that for all $\gamma\ge \gamma_s$, one has
  \bel{lhexi}
\begin{split}
\sqrt{\gamma}(\|\xi_+\|_{H_\gamma^{s-1}(\Omega_X)}+\|\zeta_+\|_{H_\gamma^{s-1}(\Omega_X)})
&\leq \frac{C(K)}{\sqrt{\gamma}}
\left(\|f^+\|_{H_\gamma^{s}(\Omega_X)}+\|f^+\|_{L^{\infty}(\Omega_X)}\|\nabla
\widetilde
\Psi_r\|_{H_\gamma^{s}(\Omega_X)}\right.\\
&\left. \hspace{-.2in}+\|U_+\|_{H_\gamma^{s}(\Omega_X)}+\|(V_r, \nabla
V_r, \nabla \widetilde
\Psi_r)\|_{H_\gamma^{s}(\Omega_X)}\|U_+\|_{W^{1,\infty}(\Omega_X)}\right).
\end{split}
\eeq
\end{lemma}

\begin{proof}
From the problem \eqref{eqxi}, we have the following
$L^2$-estimates immediately, \bel{exi}
\sqrt{\gamma}\|\xi_+\|_{L^2_\gamma(\Omega_X)} \leq
\frac{C(K)}{\sqrt{\gamma}}\|\hh_+^1\|_{L^2_\gamma(\Omega_X)}, \quad
\sqrt{\gamma}\|\zeta_+\|_{L^2_\gamma(\Omega_X)} \leq
\frac{C(K)}{\sqrt{\gamma}}\|\hh_+^2\|_{L^2_\gamma(\Omega_X)}, \eeq
with $\hh_+^1$ and $\hh_+^2$ denoting the corresponding right hand sides of the two
equations given in \eqref{eqxi}.

(1) In order to get the higher order estimates, we take derivatives
$\pd^\alpha = \pd_x^{\alpha_1}\pd_y^{\alpha_2}\pd_z^{\alpha_3}$ with
$|\alpha|=\alpha_1+\alpha_2+\alpha_3=l\leq s-1$ on both sides of the
equation of $\xi_+$ given in \eqref{eqxi} to get
\bel{xil}
\rho_r(u_r\pd_x +w_r\pd_z)\xi^{(l)}_+=\pd^\alpha
\hh_+^1 -[\partial^\alpha, \rho_r(u_r\pd_x +w_r\pd_z)]\xi_+\,,
\eeq where $\xi_+^{(l)}=\pd^\alpha \xi_+$. Using the
Gagliardo-Nirenberg inequality, we know that
\[\begin{split}
&\|\pd^\alpha \hh_+^1 -[\partial^\alpha, \rho_r(u_r\pd_x +w_r\pd_z)]\xi_+\|_{L^2_\gamma(\Omega_X)}\\
&
\hspace{.2in}
\leq C(K)\left(\|\hh_+^1\|_{H_\gamma^l(\Omega_X)}
+\|\xi_+\|_{H_\gamma^l(\Omega_X)} +\|\nabla V_r\|_{H_\gamma^{l}(\Omega_X)}\|\xi_+\|_{L^{\infty}(\Omega_X)}\right).
\end{split}
\]
Applying the same estimates as \eqref{exi} for the equation
\eqref{xil}, and using the above inequality, we obtain \bel{esti1}
\sqrt{\gamma}\|\xi_+\|_{H_\gamma^{s-1}(\Omega_X)} \leq
\frac{C(K)}{\sqrt{\gamma}}\left(\|\hh_+^1\|_{H_\gamma^{s-1}(\Omega_X)}+
\|\xi_+\|_{L^{\infty}(\Omega_X)}\|V_r\|_{H_\gamma^{s}(\Omega_X)}\right) \eeq by absorbing the term
$\frac{1}{\sqrt{\gamma}}\|\xi_+\|_{H_\gamma^{s-1}(\Omega_X)}$ in the
left hand side.

(2) Estimate of the term $\hh_+^1$. From the definition of
$\hh_+^1$, we know that \bel{esih}
\begin{split}
\|\hh_+^1\|_{H_\gamma^{s-1}(\Omega_X)}\leq&
\|(\pd_x-\frac{\pd_x\Psi_r}{\pd_y\Psi_r}\pd_y)\tilde
f^+_2\|_{H_\gamma^{s-1}(\Omega_X)}
+\|\frac{1}{\pd_y\Psi_r}\pd_y\tilde
f_1^+\|_{H_\gamma^{s-1}(\Omega_X)}
+\|R_1\cdot\nabla U_+\|_{H_\gamma^{s-1}(\Omega_X)}\\
\leq &C(K) \left\{ \|\tilde f^+\|_{H_\gamma^{s}(\Omega_X)}+ \|\nabla
\widetilde \Psi_r\|_{H_\gamma^{s}(\Omega_X)}\|\tilde
f^+\|_{L^{\infty}(\Omega_X)}
\right.\\
&\left. +\|U_+\|_{H_\gamma^{s}(\Omega_X)}+\|(V_r, \nabla
\widetilde \Psi_r)\|_{H_\gamma^{s}(\Omega_X)}\|U_+\|_{\ww(\Omega_X)}
\right\}.
\end{split}
\eeq
From \eqref{deftf}, one has \bel{esti3}
\begin{split}
\|\tilde f^+\|_{L^{\infty}(\Omega_X)}&\leq\|f^+\|_{L^{\infty}(\Omega_X)}
+C(K)\|U_+\|_{L^{\infty}(\Omega_X)} ,\\
\|\tilde f^+\|_{H_\gamma^{s}(\Omega_X)} &\leq
C(K)\left(\|f^+\|_{H_\gamma^{s}(\Omega_X)}+\|U_+\|_{H_\gamma^{s}(\Omega_X)}+
\|(V_r, \nabla V_r, \nabla \widetilde
\Psi_r)\|_{H_\gamma^{s}(\Omega_X)}\|U_+\|_{L^\infty(\Omega_X)}\right),
\end{split}
\eeq by noting from \eqref{defc} that $$\|C(U_r, \nabla U_r, \nabla
\Psi_r)\|_{H^s_\gamma(\Omega_X)}\leq C(K)\|(U_r, \nabla U_r, \nabla
\Psi_r)\|_{H_\gamma^s(\Omega_X)}.
$$
Plugging \eqref{esti3} into the
inequality \eqref{esih}, we get the estimate of $\hh_+^1$ as follows,
 \bel{esihh}
\begin{split} \|\hh_+^1\|_{H_\gamma^{s-1}(\Omega_X)} &\leq C(K)
\left\{ \| f^+\|_{H_\gamma^{s}(\Omega_X)}+\|\nabla\widetilde
\Psi_r\|_{H_\gamma^{s}(\Omega_X)}\|f^+\|_{L^{\infty}(\Omega_X)}
\right.\\
&\left.+\|U_+\|_{H_\gamma^{s}(\Omega_X)}+\|(V_r,
\nabla V_r,
\nabla \widetilde
\Psi_r)\|_{H_\gamma^{s}(\Omega_X)}\|U_+\|_{\ww(\Omega_X)} \right\}.
\end{split}
\eeq

(3) Obviously, we have  \bel{esti2}
\|\xi_+\|_{L^{\infty}(\Omega_X)}\leq C(K)\|U_+\|_{W^{1,\infty}(\Omega_X)}. \eeq
Combining \eqref{esti1}, \eqref{esihh} and \eqref{esti2}, we
obtain the following estimate of $\xi_+$ in the end, \bel{hexi}
\begin{split}
\sqrt{\gamma}\|\xi_+\|_{H_\gamma^{s-1}(\Omega_X)} \leq&
\frac{C(K)}{\sqrt{\gamma}}
\left(\|f^+\|_{H_\gamma^{s}(\Omega_X)}+\|f^+\|_{L^{\infty}(\Omega_X)}\|\nabla
\widetilde
\Psi_r\|_{H_\gamma^{s}(\Omega_X)}\right.\\
&\left. +\|U_+\|_{H_\gamma^{s}(\Omega_X)}+\|(V_r, \nabla
V_r, \nabla \widetilde
\Psi_r)\|_{H_\gamma^{s}(\Omega_X)}\|U_+\|_{W^{1,\infty}(\Omega_X)}\right).
\end{split}
\eeq
One can study $\zeta_+$ similarly from the problem \eqref{eqxi}, and deduce the same estimate as \eqref{hexi} for  $\zeta_+$. This finishes the proof of this lemma. \end{proof}

 Similar to
\eqref{defxi}, for $U_-=(u_-, v_-, w_-, p_-)^T$, we define the
vorticities \bel{defxi-} \xi_-=(\pd_x
-\frac{\pd_x\Psi_l}{\pd_y\Psi_l}\pd_y)v_-
-\frac{1}{\pd_y\Psi_l}\pd_y u_-, \quad
\zeta_-=(\pd_z-\frac{\pd_z\Psi_l}{\pd_y\Psi_l}\pd_y)v_--\frac{1}{\pd_y\Psi_l}
\pd_y w_-. \eeq From the equations given in \eqref{lieqr}, we deduce
that $(\xi_-, \zeta_-)$ also satisfy two transport equations similar
to \eqref{eqxi}, and conclude

\begin{lemma} Let $s>1$, there exist constants $C(K)>0$ and $\gamma_s\ge1$ such
that for all $\gamma\ge \gamma_s$, the following estimate holds,
  \bel{hexi-}
\begin{split}
\sqrt{\gamma}(\|\xi_-\|_{H_\gamma^{s-1}(\Omega_X)}+\|\zeta_-\|_{H_\gamma^{s-1}(\Omega_X)}
)&\leq \frac{C(K)}{\sqrt{\gamma}}
\left(\|f^-\|_{H_\gamma^{s}(\Omega_X)}+\|f^-\|_{L^{\infty}(\Omega_X)}
\|\nabla \widetilde\Psi_l\|_{H_\gamma^{s}(\Omega_X)}\right.\\
&\left. \hspace{-.2in}+\|U_-\|_{H_\gamma^{s}(\Omega_X)}+\|(V_l, \nabla
V_l, \nabla \widetilde
\Psi_l)\|_{H_\gamma^{s}(\Omega_X)}\|U_-\|_{W^{1,\infty}(\Omega_X)}\right).
\end{split}
\eeq
\end{lemma}

The estimates \eqref{lhexi} and \eqref{hexi-} will be used to
study the normal derivatives of $W^c=(W_1^\pm, W_2^\pm)^T$ in next subsection.

\subsubsection{Estimate of normal derivatives}
After studying the ``vorticities" $(\xi_\pm, \zeta_\pm)^T$ in the previous
subsection, we try to represent $\pd_yW^c=(\pd_y W_1^+,  \pd_y
W_2^+, \pd_y W_1^-, \pd_y W_2^-)^T$ in terms of $(\xi_{\pm},
\zeta_\pm)^T$ and $\pd_T W$.

From the transformations defined in \eqref{utow}, we get \bel{uvw}
\begin{cases}
\pd_y u_+=\pd_y W^+_1 -\pd_x\Psi_r(\pd_y W^+_3 +\pd_y W_4^+) +(\pd_yT^{-1}_rW^+)_1,\\
\pd_y v_+=\pd_x\Psi_r\pd_y W^+_1 +\pd_z\Psi_r\pd_yW^+_2+\pd_y W^+_3
+\pd_y W_4^+
 +(\pd_yT^{-1}_rW^+)_2,\\
\pd_y w_+=\pd_y W^+_2 -\pd_z\Psi_r(\pd_y W^+_3 +\pd_y W_4^+) +(\pd_yT^{-1}_rW^+)_3,\\
\end{cases}
\eeq which implies that
\[
\begin{split}
&(\pd_x v_+-\xi_+)\pd_y\Psi_r=(1+(\pd_x\Psi_r)^2)\pd_y W^+_1 +\pd_x\Psi_r\pd_z\Psi_r\pd_y W_2^+ +(\pd_y T_r^{-1}W^+)_1 +\pd_x\Psi_r(\pd_yT_r^{-1}W^+)_2\\
&(\pd_z v_+-\zeta_+)\pd_y\Psi_r=\pd_x\Psi_r\pd_z\Psi_r\pd_y W^+_1
+(1+(\pd_z\Psi_r)^2)\pd_y W_2^+ +(\pd_y T_r^{-1}W^+)_3
+\pd_z\Psi_r(\pd_yT_r^{-1}W^+)_2,
\end{split}
\]
with $(\pd_y T_r^{-1}W^+)_i ~(i=1,2,3)$ representing the $i$-th
component of the vector $\pd_y T_r^{-1}W^+$, and $T_{r,l}=T(\nabla \Psi_{r,l})$. Thus, we obtain
\bel{yc1}
\begin{split}
\pd_y W_1^+=&\frac{1}{\lc\pd\Psi_r\rc^2}\left\{\pd_y\Psi_r[(1+(\pd_z\Psi_r)^2)(\pd_x v_+-\xi_+)-\pd_x\Psi_r\pd_z\Psi_r(\pd_z v_+-\zeta_+)]\right.\\
&\left.-(1+(\pd_z\Psi_r)^2)(\pd_y T_r^{-1}W^+)_1-\pd_x\Psi_r(\pd_y
T_r^{-1}W^+)_2+\pd_x\Psi_r\pd_z\Psi_r(\pd_y T_r^{-1}W^+)_3\right\},
\end{split}
\eeq and \bel{yc2}
\begin{split}
\pd_y W_2^+=&\frac{1}{\lc\pd\Psi_r\rc^2}\left\{\pd_y\Psi_r[(1+(\pd_x\Psi_r)^2)(\pd_z v_+-\zeta_+)-\pd_x\Psi_r\pd_z\Psi_r(\pd_x v_+-\xi_+)]\right.\\
&\left.+\pd_x\Psi_r\pd_z\Psi_r(\pd_y
T_r^{-1}W^+)_1-\pd_z\Psi_r(\pd_y
T_r^{-1}W^+)_2-(1+(\pd_x\Psi_r)^2)(\pd_y T_r^{-1}W^+)_3\right\}.
\end{split}
\eeq

From \eqref{yc1}, \eqref{yc2} and the equations of $W^+$
given in \eqref{pb}, we can represent $\pd_y W^+$ by  $\xi_+, \zeta_+$ and $\pd_T W^+$ as follows \bel{eqn} \begin{split} \pd_y
W^+=&\Lambda_2F^+ +\tilde A_1^r \pd_xW^+ +\tilde A_3^r\pd_z W^+
+\tilde A_0^r W^+\\
&+\frac{\pd_y \Psi_r}{\lc \pd\Psi_r\rc^2}\begin{pmatrix}
  -(1+(\pd_z\Psi_r)^2) & \pd_x\Psi_r\pd_z\Psi_r\\
  \pd_x\Psi_r\pd_z\Psi_r & -(1+(\pd_x\Psi_r)^2)\\
  0 & 0\\
  0 &0
\end{pmatrix}
\begin{pmatrix}
  \xi_+\\ \zeta_+
\end{pmatrix},\end{split}
\eeq where $\tilde A_0^r,\tilde A_{1}^r$ and $\tilde A_3^r$ are
modifications of $A_0^rC^r, A_{1}^r$ and $A_3^r$ given in
\eqref{vdefac} after adding equations \eqref{yc1} and \eqref{yc2}.
Similarly, from the formulae of $\xi_-, \zeta_-$ and the equations of $W^-$ given in \eqref{pb}, one can derive a representation of $\pd_yW^-$ similar to that given in \eqref{eqn}.

By using the equation \eqref{eqn} of $\pd_y W^\pm$, we get the
following result of the normal derivative $\pd_yW$:
\begin{lemma}\label{len}
For $s\ge 1$, there exists a constant $C(K)>0$ such that for any
integer $k ~(1\le k\le s)$, the following estimate holds, \bel{nestimate}
\begin{split}
\|\pd_y^kW\|_{L_y^2(H_\gamma^{s-k})} \leq &
C(K)\left\{\|F\|_{H_\gamma^{s-1}(\Omega_X)}+\|\xi\|_{H_\gamma^{s-1}(\Omega_X)}
+\|\zeta\|_{H_\gamma^{s-1}(\Omega_X)}+\|W\|_{L_y^2(H_\gamma^{s})}
+\|W\|_{H_\gamma^{s-1}(\Omega_X)}
\right.\\
&\left.+(\|\xi\|_{L^\infty(\Omega_X)}+\|\zeta\|_{L^\infty(\Omega_X)})
\|\nabla\widetilde\Psi\|_{H_\gamma^{s-1}(\Omega_X)}
+\|W\|_{L^\infty(\Omega_X)}\|(V, \nabla \widetilde \Psi)
\|_{H_\gamma^{s}(\Omega_X)}\right\}\\
\end{split}
\eeq
where $\xi=(\xi_+, \xi_-)^T, \zeta=(\zeta_+, \zeta_-)^T$ are defined in \eqref{defxi} and \eqref{defxi-}
respectively.
\end{lemma}

\begin{proof}
We shall only study the estimate of $W^+$ by
induction on $k$, and the estimate of $W^-$ can be derived similarly.

(1) When $k=1$, we study the estimate of $\|\pd_y
W^+\|_{L^2(H_\gamma^{s-1})}$ through the equations \eqref{eqn}.
As in the proof of Lemma \ref{let}, using Gagliardo-Nirenberg's
inequality we obtain that
\[\begin{split}
&\|\tilde A_1^r \pd_x W^++\tilde A_3^r
\pd_z W^+\|_{L_y^2(H_\gamma^{s-1})}\leq C(K)
\left(\|W^+\|_{L_y^2(H_\gamma^s)}+\|W^+\|_{L^\infty(\Omega_X)}
\|(V_r, \nabla \widetilde \Psi_r)\|_{L_y^2(H_\gamma^{s})}\right),\\
&\|\tilde A_0^r W^+\|_{L_y^2(H_\gamma^{s-1})}\leq C(K)
\left(\|W^+\|_{L_y^2(H_\gamma^{s-1})}+\|W^+\|_{L^\infty(\Omega_X)}\|(V_r, \nabla V_r, \nabla \widetilde \Psi_r, \nabla^2 \widetilde \Psi_r)\|_{L_y^2(H_\gamma^{s-1})}\right)
\end{split}
\]
by noting that $\tilde A_0^r$ is a $C^\infty$ function of
$(V_r, \nabla V_r, \nabla \widetilde\Psi_r,
\nabla^2 \widetilde \Psi_r)$ and vanishes at the origin.
Moreover, the estimate for the terms of $(\xi_+,
\zeta_+)$ appeared in \eqref{eqn} is in the following,
\[
\begin{array}{l}
\|\frac{\pd_y
\Psi_r}{\lc
\pd\Psi_r\rc^2}(\pd_x\Psi_r\pd_z\Psi_r\zeta_+-(1+(\pd_z\Psi_r)^2)\xi_+)\|
_{L_y^2(H_\gamma^{s-1})}\\[2mm] +\|\frac{\pd_y \Psi_r}{\lc
\pd\Psi_r\rc^2}(\pd_x\Psi_r\pd_z\Psi_r\xi_+-(1+(\pd_x\Psi_r)^2)\zeta_+)\|
_{L_y^2(H_\gamma^{s-1})}
\\[2mm]
\hspace{.5in}
\le
C(K)\left(\|\xi_+, \zeta_+\|_{L_y^2(H_\gamma^{s-1})}+\|\xi_+, \zeta_+\|_{L^\infty(\Omega_X)}\|\nabla
\widetilde\Psi_r\|_{L_y^2(H_\gamma^{s-1})}\right).
\end{array}
\]


Thus, from the equations of $\pd_yW^+$ given in
\eqref{eqn}, we get
\[
\begin{split}
  \|\pd_yW^+\|_{L_y^2(H_\gamma^{s-1})}
  \leq
& C_1(K)\left\{\|F^+\|_{H_\gamma^{s-1}(\Omega_X)}+\|\xi_+\|_{H_\gamma^{s-1}
  (\Omega_X)}
  +\|\zeta_+\|_{H_\gamma^{s-1}(\Omega_X)}+\|W^+\|_{L_y^2(H_\gamma^{s})}
\right.\\
  &\left.+(\|\xi_+\|_{L^\infty(\Omega_X)}+\|\zeta_+\|_{L^\infty(\Omega_X)})\|\nabla\widetilde \Psi_r\|_{H_\gamma^{s-1}(\Omega_X)}+\|W^+\|_{L^\infty(\Omega_X)}\|(V_r, \nabla \widetilde \Psi_r)
  \|_{H_\gamma^{s}(\Omega_X)}\right\},
\\
\end{split}
\]
which implies the estimate \eqref{nestimate} for the case $k=1.$

(2) Assuming that the estimate \eqref{nestimate} holds for $k-1$, we try to prove that it also holds for $k\leq s$. By taking derivatives with respect to $y$  on the equation \eqref{eqn}, we get
\bel{eqk} \pd_y^k W^+=\Lambda_2\pd_y^{k-1}F^+ +\pd_y^{k-1}\left[\tilde
A_1^r \pd_xW^+ +\tilde A_3^r\pd_z W^+ +\tilde A_0^r W^++S
\begin{pmatrix}
  \xi_+\\ \zeta_+
\end{pmatrix}\right],
\eeq with $S$ denoting the zero-th order coefficient matrix of
$(\xi_+, \zeta_+)^T$ given in \eqref{eqn}. Obviously, we have
\bel{3-57-1}
\|\Lambda_2 \pd_y^{k-1}F^+\|_{L_y^2(H_\gamma^{s-k})}\leq
\|F^+\|_{H_\gamma^{s-1}(\Omega_X)}.
\eeq

To estimate $\|\pd_y^{k-1}(\tilde
A_1^r\pd_xW^+)\|_{L^2(H_\gamma^{s-k})}$, we define
$\pd_T^\alpha=\pd_x^{\alpha_1}\pd_z^{\alpha_2}$ with
$\alpha_1+\alpha_2=l\leq s-k$,
and get
\bel{pdyk}
\begin{split}
\pd_T^\alpha \pd_y^{k-1}(\tilde A_1^r \pd_xW^+)=&\tilde A_1^r
\pd_T^\alpha \pd_y^{k-1}\pd_xW^+ +[\pd_T^\alpha \pd_y^{k-1}, \tilde A_1^r ]\pd_xW^+
\end{split} \eeq
and
\[\begin{split}
\|\tilde A_1^r \pd_T^\alpha
\pd_y^{k-1}\pd_xW^+\|_{L^2_\gamma(\Omega_X)}&\leq C(K)
\|\pd_y^{k-1}W^+\|_{L_y^2(H_\gamma^{l+1})},\\
\|[\pd_T^\alpha \pd_y^{k-1}, \tilde A_1^r ]\pd_xW^+\|_{L^2_\gamma(\Omega_X)}&\leq
C(K)\left(\|W^+\|_{H_\gamma^{l+k-1}(\Omega_X)}+\|W^+\|_{L^\infty(\Omega_X)}
\|(V_r, \nabla \widetilde \Psi_r)\|_{H_\gamma^{l+k}(\Omega_X)}\right).
\end{split}
\]
Substituting the above estimates into \eqref{pdyk} and taking
weighted summation from $l=0$ to $l=s-k$, we obtain that
\bel{pdy1}\begin{split}
&\|\pd_y^{k-1}(\tilde A_1^r \pd_xW^+)\|_{L_y^2(H_\gamma^{s-k})}\\
\leq &C(K)\left(\|\pd_y^{k-1}W^+\|_{L_y^2(H_\gamma^{s-(k-1)})}
+\|W^+\|_{H_\gamma^{s-1}(\Omega_X)}+\|W^+\|_{L^\infty(\Omega_X)}\|(V_r, \nabla \widetilde \Psi_r)\|_{H_\gamma^{s}(\Omega_X)}\right).
\end{split}
\eeq

In the same way, we deduce that $\|\pd_y^{k-1}(\tilde A_3^r
\pd_zW^+)\|_{L_y^2(H_\gamma^{s-k})}$ is also bounded by the
right hand side of \eqref{pdy1}. Similarly, we can get
\bel{3-61}
\|\pd_y^{k-1}(\tilde A_0^r W^+)\|_{L_y^2(H_\gamma^{s-k})}\leq C(K)\left(
\|W^+\|_{H_\gamma^{s-1}(\Omega_X)}+\|W^+\|_{L^\infty(\Omega_X)}\|(V_r, \nabla \widetilde \Psi_r)\|_{H_\gamma^{s}(\Omega_X)}\right),\eeq
and
\bel{3-62}
\begin{split}
\|\pd_y^{k-1}(S\begin{pmatrix}
  \xi_+\\ \zeta_+
\end{pmatrix})\|_{L_y^2(H_\gamma^{s-k})}&\leq C(K)\left((\|\xi_+\|_{L^\infty(\Omega_X)}+\|\zeta_+\|_{L^\infty(\Omega_X)})\|\nabla\widetilde\Psi_r\|_{H_\gamma^{s-1}(\Omega_X)}
\right.\\
&\qquad\qquad\left.+\|\xi_+\|_{H_\gamma^{s-1}(\Omega_X)}+\|\zeta_+\|_{H_\gamma^{s-1}(\Omega_X)}
\right).
\end{split}
\eeq
Combining the inequalities \eqref{3-57-1}, \eqref{pdy1}, \eqref{3-61}, \eqref{3-62} with
\eqref{eqk}, and using the induction assumption of $\partial_y^{k-1}W^+$, we obtain the estimate \eqref{nestimate} for all $k\leq
s$ and conclude this lemma. \end{proof}

By plugging the estimates of $\xi, \zeta$ given in \eqref{lhexi} and
\eqref{hexi-} into the inequality \eqref{nestimate} and taking summation
from $k=1$ to $k=s$, one deduces

\begin{lemma}
Let $s\ge 1$, there exist constants $C(K)$ and $\gamma_s\geq 1$
such that for all $\gamma \ge \gamma_s$, the following inequality
holds,\bel{norestimate}
\begin{split}
\sqrt{\gamma}\sum\limits_{k=1}\limits^{s}\|\pd_y^kW\|_{L_y^2(H_\gamma^{s-k})}
\leq
&
C(K)\left\{\sqrt{\gamma}\left(\|F\|_{H_\gamma^{s-1}(\Omega_X)}+\|W\|_{L_y^2(H_\gamma^{s})}
+\|W\|_{H_\gamma^{s-1}(\Omega_X)}\right.\right.\\
&\left.+\|W\|_{L^\infty(\Omega_X)}  \|(V, \nabla
\widetilde \Psi)\| _{H_\gamma^{s}(\Omega_X)}\right)
+\frac{1}{\sqrt{\gamma}}\left(\|f\|_{H_\gamma^{s}(\Omega_X)}
+\|U\|_{H_\gamma^{s}(\Omega_X)}\right.\\
&+\|f\|_{L^\infty(\Omega_X)}\|\nabla
\widetilde\Psi\|_{H_\gamma^{s}(\Omega_X)}+\|f\|_{\ww(\Omega_X)}\|\nabla
\widetilde\Psi\|_{H_\gamma^{s-1}(\Omega_X)}
\\
&\left.\left.+\|(V, \nabla
V, \nabla
\widetilde \Psi)\|_{H_\gamma^{s}(\Omega_X)}\|W\|_{\ww(\Omega_X)}\right)\right\}.\end{split}
\eeq
\end{lemma}

\subsubsection{Proof of Theorem \ref{hsestimate}}
After having the estimates given in \eqref{testimate} and
\eqref{norestimate} on tangential derivatives and normal derivatives
of solutions to the problem \eqref{pb}, we are going to prove the
estimate \eqref{test} given in Theorem
\ref{hsestimate}.

From the definition of the space $H^s_\gamma(\Omega_X)$,
\[
\|W\|_{H^s_\gamma(\Omega_X)}=\sum\limits_{k=0}\limits^s\|\pd_y^kW\|
_{L_y^2(H_\gamma^{s-k})},
\]
and combining estimates \eqref{norestimate} and
\eqref{testimate}, we  obtain that
\bel{hest}
\begin{split}
&\sqrt{\gamma}\|W\|_{H_\gamma^s(\Omega_X)} +\|
W^{nc}_{|y=0}\|_{H_\gamma^s(\omega_X)} +\|\phi\|_{H_\gamma^{s+1}(\omega_X)}
\leq
C(K)\left\{\frac{1}{\sqrt{\gamma}}\|F\|_{H_\gamma^{s}(\Omega_X)}+
\frac{1}{\gamma^{3/2}}\|F\|_{L_y^2(H_\gamma^{s+1})}\right.\\
&\hspace{.8in}+\frac{1}{\gamma}\|g\|_{H_\gamma^{s+1}(\omega_X)}+\frac{1}{\sqrt{\gamma}}\|W\|_
{H_\gamma^s(\Omega_X)}
+\frac{1}{\gamma^{3/2}}\|W\|_{L^{\infty}(\Omega_X)}\|(V, \nabla \widetilde
\Psi)\|_{H_\gamma^{s+2}(\Omega_X)}\\
&\hspace{.8in}+\frac{1}{\gamma}\left(\|W^{nc}_{|y=0}\|_ {L^\infty(\omega_X)}
+\|\phi\|_{W^{1,\infty}(\omega_X)}\right)\| (V, \pd_y
V, \nabla \widetilde \Psi)_{|y=0}\|_{H_\gamma^{s}(\omega_X)}\\
& \hspace{.8in}+\frac{1}{\sqrt{\gamma}}\left(\|f\|_{H_\gamma^{s}(\Omega_X)}
+\|f\|_
{L^\infty(\Omega_X)}\|\nabla
\widetilde\Psi\|_{H_\gamma^{s}(\Omega_X)}
+\|f\|_
{W^{1,\infty}(\Omega_X)}\|\nabla
\widetilde\Psi\|_{H_\gamma^{s-1}(\Omega_X)}\right)\\
&\hspace{.8in}\left.+\frac{1}{\sqrt{\gamma}}\left(\|U\|_{H_\gamma^{s}(\Omega_X)}+\|(V, \nabla V, \nabla \widetilde
\Psi)\|_{H_\gamma^{s}(\Omega_X)}\|W\|_{\ww(\Omega_X)}\right)\right\}.
\end{split}
\eeq

(1)
From the definition of $F$ given in \eqref{pb}, we have
\bel{Ff}\begin{split}
&\frac{1}{\sqrt{\gamma}}\|F\|_{H_\gamma^{s}(\Omega_X)}+
  \frac{1}{\gamma^{3/2}}\|F\|_{L_y^2(H_\gamma^{s+1})}\\
  \leq &C(K)\left( \frac{1}{\sqrt{\gamma}}\|f\|_{H_\gamma^{s}(\Omega_X)}+
  \frac{1}{\gamma^{3/2}}\|f\|_{L_y^2(H_\gamma^{s+1})}
  +\frac{1}{\gamma^{3/2}}\|f\|_{L^\infty(\Omega_X)}\|\nabla \widetilde \Psi\|_{H_\gamma^{s+1}(\Omega_X)}\right).
\end{split}
\eeq

From the transformation \eqref{utow} between $U$ and $W$, we
know that 
\begin{align}
&\|W\|_{L^\infty(\Omega_X)} \leq
C(K)\|U\|_{L^\infty(\Omega_X)}, \quad \|W\|_{\ww(\Omega_X)} \leq
C(K)\|U\|_{\ww(\Omega_X)},\label{wu1}  \\
&\|U\|_{H_\gamma^s(\Omega_X)}\leq
C(K)\left(\|W\|_{H_\gamma^s(\Omega_X)}+\|W\|_{L^\infty(\Omega_X)}\|\nabla
\widetilde\Psi\|_{H_\gamma^s(\Omega_X)}\right),\label{wu2}\\
&\|\PP U|_{y=0}\|_{H_\gamma^s(\omega_X)}\leq C(K)\left(\|
W^{nc}_{|y=0}\|_{H_\gamma^s(\omega_X)} +\|\PP
U|_{y=0}\|_{L^\infty(\omega_X)}\| \nabla
\widetilde\Psi_{|y=0}\|_{H_\gamma^s(\omega_X)}\right).\label{wu3}
\end{align}

By plugging the inequalities \eqref{Ff}, \eqref{wu1}, \eqref{wu2} and \eqref{wu3} into the right hand side of \eqref{hest}, and absorbing the term
$\frac{1}{\sqrt{\gamma}}\|W\|_{H_\gamma^s(\Omega_X)}$ by the
left hand side of \eqref{hest}, we deduce that
\bel{hu}
\begin{split}
&\sqrt{\gamma}\|U\|_{H_\gamma^s(\Omega_X)} +\|\PP
U|_{y=0}\|_{H_\gamma^s(\omega_X)}
+\|\phi\|_{H_\gamma^{s+1}(\omega_X)} \leq
C(K)\left\{\frac{1}{\gamma^{3/2}}\|f\|_{H_\gamma^{s+1}(\Omega_X)}
+\frac{1}{\gamma}\|g\|_{H_\gamma^{s+1}(\omega_X)}\right.\\
&\hspace{.8in}+\frac{1}{\gamma^{3/2}}\|f\| _{L^\infty(\Omega_X)}\|
\nabla \widetilde \Psi\|_{H_\gamma^{s+1}
(\Omega_X)}+\frac{1}{\sqrt{\gamma}}\|f\| _{\ww(\Omega_X)}\|
\nabla \widetilde \Psi\|_{H_\gamma^{s-1}
(\Omega_X)}\\
&
\hspace{.8in}+\frac{1}{\gamma^{3/2}}\|U\|_{L^{\infty}(\Omega_X)}\|(V, \nabla \widetilde
\Psi)\|_{H_\gamma^{s+2}(\Omega_X)}+\frac{1}{\sqrt{\gamma}}
\|U\|_{\ww(\Omega_X)}
\|(V, \nabla V, \nabla \widetilde
\Psi)\|_{H_\gamma^{s}(\Omega_X)}\\
&\hspace{.8in}
\left.+\frac{1}{\gamma}\left(\|\PP U|_{y=0}\|_{L^\infty(\omega_X)}
+\|\phi\|_{W^{1,\infty}(\omega_X)}\right)\|(V, \pd_y
V, \nabla \widetilde
\Psi)_{|y=0}\|_{H_\gamma^{s}(\omega_X)}
\right\}.
\end{split}
\eeq

(2) To conclude the estimate \eqref{test} from \eqref{hu}, the main remaining task is to control the terms with $L^\infty$
norm and $\ww$ norm on the right hand side of \eqref{hu}. Obviously, one has
\bel{ws1}
\begin{split}
\|f\|_{L^\infty(\Omega_X)} \leq \frac{Ce^{\gamma
X}}{\sqrt{\gamma}}\|f\|_{H_\gamma^2(\Omega_X)}, \quad
\|U\|_{\ww(\Omega_X)}\leq Ce^{\gamma
X}\|U\|_{H_\gamma^3(\Omega_X)}, \\
 \|\PP U|_{y=0}\|_{L^\infty(\omega_X)} \leq \frac{Ce^{\gamma
X}}{\gamma}\|\PP U|_{y=0}\|_{H_\gamma^2(\omega_X)}, \quad
\|\phi\|_{\ww(\omega_X)}\leq Ce^{\gamma
X}\|\phi\|_{H_\gamma^3(\omega_X)}. \end{split}\eeq 
Using the above estimates in \eqref{hu} and setting $s=3$, we get
\bel{estt}
\begin{split}
&\sqrt{\gamma}\|U\|_{H_\gamma^3(\Omega_X)} +\|\PP
U|_{y=0}\|_{H_\gamma^3(\omega_X)} +\|\phi\|_{H_\gamma^{4}(\omega_X)}
\leq C(K)\left\{
\frac{1}{\gamma^{3/2}}\|f\|_{H_\gamma^{4}(\Omega_X)}
+\frac{1}{\gamma}\|g\|_{H_\gamma^{4}(\omega_X)}\right.\\
&\hspace{.8in}
+\frac{e^{\gamma X}}{\gamma^{2}}
\|f\|_{H_\gamma^2(\Omega_X)}\|\nabla \widetilde
\Psi\|_{H_\gamma^{4} (\Omega_X)}
+\frac{e^{\gamma X}}{\sqrt{\gamma}}\|f\| _{H^3_\gamma(\Omega_X)}\|
\nabla \widetilde \Psi\|_{H_\gamma^{2}(\Omega_X)}\\
& \hspace{.8in}
+\frac{e^{\gamma
X}}{\gamma^2}\|U\|_{H_\gamma^{3}(\Omega_X)}\|(V, \nabla
\widetilde \Psi)\|_{H_\gamma^{5}(\Omega_X)}+
\frac{e^{\gamma X}}{\sqrt{\gamma}}
\|U\|_{H_\gamma^3(\Omega_X)}
\|(V, \nabla V, \nabla \widetilde
\Psi)\|_{H_\gamma^{3}(\Omega_X)}\\
&\hspace{.8in}
\left.+\left(\frac{e^{\gamma X}}{\gamma^2}\|\PP
U|_{y=0}\|_{H_\gamma^2(\omega_X)} +\frac{e^{\gamma
X}}{\gamma}\|\phi\|_{H_\gamma^3(\omega_X)}\right)\| (V, \pd_y V, \nabla \widetilde
\Psi)_{|y=0}\|_{H_\gamma^{3}(\omega_X)}\right\}.
\end{split}
\eeq

Under the assumption \eqref{asmp} given in Theorem
\ref{hsestimate},
\[
\|\nabla \widetilde
\Psi\|_{H_\gamma^{5}(\Omega_X)}+\|V\|_{H_\gamma^{5}(\Omega_X)}\leq
K,
\]
one can eliminate the terms
$\|U\|_{H_\gamma^3(\Omega_X)}$, $\|\PP
U|_{y=0}\|_{H_\gamma^2(\omega_X)}$ and
$\|\phi\|_{H_\gamma^{3}(\omega_X)}$ on the right hand side of
\eqref{estt} by fixing $\gamma$ large enough, and concludes
\bel{estt-1}
\begin{split}
&\sqrt{\gamma}\|U\|_{H_\gamma^3(\Omega_X)} +\|\PP
U|_{y=0}\|_{H_\gamma^3(\omega_X)} +\|\phi\|_{H_\gamma^{4}(\omega_X)}\\
&\hspace{.8in}\leq C(K)\left\{
\frac{1}{\gamma^{3/2}}\|f\|_{H_\gamma^{4}(\Omega_X)}
+\frac{1}{\gamma}\|g\|_{H_\gamma^{4}(\omega_X)}+\frac{e^{\gamma X}}{\gamma^{2}}
\|f\|_{H_\gamma^2(\Omega_X)}
+\frac{e^{\gamma X}}{\sqrt{\gamma}}\|f\| _{H^3_\gamma(\Omega_X)}\right\},
\end{split}
\eeq
which implies
\[
\|U\|_{W^{1,\infty}(\Omega_X)}+\|\PP U|_{y=0}\|_{L^\infty(\omega_X)}
  +\|\phi\|_{W^{1,\infty}(\omega_X)}\leq C(K)(\|f\|_{H^4(\Omega_X)}
  +\|g\|_{H^4(\omega_X)})
\]
by using \eqref{ws1}, and fixing a large $\gamma>0$ with $\gamma X\le 1$.

Substituting the above inequality into \eqref{hu}, it follows
\bel{hu-1}
\begin{split}
&\sqrt{\gamma}\|U\|_{H_\gamma^s(\Omega_X)} +\|\PP
U|_{y=0}\|_{H_\gamma^s(\omega_X)}
+\|\phi\|_{H_\gamma^{s+1}(\omega_X)} \leq
C(K)\left\{\frac{1}{\gamma^{3/2}}\|f\|_{H_\gamma^{s+1}(\Omega_X)}
+\frac{1}{\gamma}\|g\|_{H_\gamma^{s+1}(\omega_X)}\right.\\
&\hspace{.3in}+\frac{1}{\gamma^{3/2}}\|f\| _{L^\infty(\Omega_X)}\|
\nabla \widetilde \Psi\|_{H_\gamma^{s+1}
(\Omega_X)}+\frac{1}{\sqrt{\gamma}}\|f\| _{\ww(\Omega_X)}\|
\nabla \widetilde \Psi\|_{H_\gamma^{s-1}
(\Omega_X)}\\
&
\hspace{.3in}+\left.\left(\frac{1}{\gamma^{3/2}}\|(V, \nabla \widetilde
\Psi)\|_{H_\gamma^{s+2}(\Omega_X)}+\frac{1}{\gamma}\|\pd_y
V_{|y=0}\|_{H_\gamma^{s}(\omega_X)}\right)(\|f\|_{H^4(\Omega_X)}
  +\|g\|_{H^4(\omega_X)})
\right\}.
\end{split}
\eeq
This completes the proof of the tame estimate given in Theorem
\ref{hsestimate}.

\vspace{.1in}
By fixing a large $\gamma>0$ and then $\gamma X\le 1$ in Theorem
\ref{hsestimate}, we immediately obtain,

\begin{coro}\label{coro4.7}
For any $s\ge 0$, assume that the non-planar contact discontinuity
given in \eqref{vbstate} satisfies \eqref{2.15}, \eqref{3.3} and
$(V_{r,l}, \nabla \widetilde \Psi_{r,l})\in
H^{s+2}(\Omega_X)\cap H^{5}(\Omega_X)$ with
\bel{asmp-1}
\|\nabla
\widetilde \Psi\|_{H^{5}(\Omega_X)}+\|
V\|_{H^{5}(\Omega_X)}\leq K. \eeq
Then, for the linear problem \eqref{lpb}, there exists a constant
$K_0>0$,  for all $K\leq K_0$, there is a constant $C(K,s)$ depending on $K$ and $s$, such that
for all $(U, \phi)\in
(H^{s+2}(\Omega_X)\times H^{s+2}(\omega_X))\cap (H^{5}(\Omega_X)\times H^{5}(\omega_X))$, one has
\bel{test-1}
\begin{split}
&\|U\|_{H^s(\Omega_X)} +\|\PP
U|_{y=0}\|_{H^s(\omega_X)}
+\|\phi\|_{H^{s+1}(\omega_X)} \\
&
\hspace{0.6in}\leq
C(K,s)\left\{\|f\|_{H^{s+1}(\Omega_X)}
+\|g\|_{H^{s+1}(\omega_X)}+\|(V, \nabla \widetilde
\Psi)\|_{H^{s+2}(\Omega_X)}(\|f\|_{H^4(\Omega_X)}
  +\|g\|_{H^4(\omega_X)})
\right\}.
\end{split}
\eeq
\end{coro}

\section{Iteration scheme}
\setcounter{equation}{0}

The remainder of this work is to obtain the existence of solutions to the nonlinear
problem \eqref{2.11} by constructing a proper iteration scheme.  From the estimate \eqref{test-1},  we know that there is loss of regularity of solutions to the linearized problem \eqref{lpb} with respect to $f$ and $g$, so as in \cite{Chen-Wang, cou2,  tra1}, we shall adapt the Nash-Moser-H\"ormander iteration  scheme to study the nonlinear problem \eqref{2.11}

\subsection{Compatibility conditions and the zero-order approximate solution}

Given initial data $(U_0^\pm(y,z), \psi_0(z))$ on $\{x=0\}$ with
$U_0^\pm(y,z)=\overline{U}_{r,l}+\widetilde
U_0^\pm(y,z)$, $\psi_0(z)=\psi_0(0,z)$ satisfying $\widetilde
U_0^\pm\in
H^{s+\frac{1}{2}}(\R_+^2)$, $\psi_0\in H^{s+1}(\R)$ for a fixed
$s\ge 3$, and
\begin{equation}
\text{Supp} ~\widetilde U_0^\pm\subseteq\{y\ge0, y^2+z^2\le
1\},\quad \text{Supp}~ \psi_0\subseteq[-1,1],
\end{equation} we first
state the compatibility conditions for the existence of a
classical solution to the nonlinear problem \eqref{2.11}.

As in \eqref{2.10}, first we extend $\psi_0(z)$ to $\tilde \Psi_0^\pm(y,z)$ supported in
$\{y\ge 0, \sqrt{y^2+z^2}\le 1+C(X)\}$ with $C(X)$ being a function
of $X$ such that $\tilde \Psi_0^\pm\in H^{s+\frac{3}{2}}(\R_+^2)$ with
\begin{equation}
  \|\tilde \Psi_0^\pm\|_{H^{s+\frac{3}{2}}(\R_+^2)}\le
  C\|\psi_0\|_{H^{s+1}(\R)},
\end{equation}
and $\Psi_0^\pm(y,z)=\pm y+\tilde \Psi_0^\pm$
satisfy
\begin{equation}
  \pd_y\Psi_0^+\ge \frac{5}{6}, \quad \pd_y \Psi_0^-\le -\frac{5}{6}
\qquad {\rm for~all}~ y>0.
\end{equation}

Set
$$\widetilde U^\pm(x,y,z)=U^\pm(x,y,z)-\overline{U}_{r,l}, \quad
\tilde \Psi^\pm(x,y,z)=\Psi^\pm(x,y,z)\mp y. $$
From the equations of $\Psi^\pm$ and $U^\pm$ given in \eqref{2.10}
and \eqref{2.11}, we can determine $\pd_x^{j+1}\tilde \Psi^\pm$ and
$\pd_x^{j+1} \widetilde U^\pm$ on $\{x=0\}$ by induction on $j\in\N$ in the following,
\begin{equation}\label{psi-j}
\pd_x^{j+1} \Psi^\pm=\pd_x^j\left(\frac{1}{u^\pm}
(v^\pm-\pd_z \Psi^\pm w^\pm)\right)
\end{equation}
and
\begin{equation}\label{u-j}
  \pd_x^{j+1} U^\pm=\pd_x^j\left(-A_1^{-1}\left(\frac{1}{ \Psi^\pm_y}(A_2-\Psi_x^\pm
A_1-\Psi_z^\pm A_3)\pd_y U^\pm+A_3\pd_z
U^\pm\right)\right).
\end{equation}
Thus, for a fixed $k\le s$, the data $(U_0^\pm, \Psi_0^\pm)$ are compatible up to order
$k$ for the problem \eqref{2.10} and \eqref{2.11}, if 
\begin{equation}\label{defj}
\begin{split}
  &\pd_x^j \Psi^+=\pd_x^j \Psi^-,\qquad \pd_x^j p^+=\pd_x^j p^-
\end{split}
\end{equation}
hold for all $0\le j\le k$ at $\{y=0\}\cap\{x=0\}$.

From now on, we assume the following hypothesis:
\begin{quote}
\item[(H)] for a fixed $s>\frac{27}{2}$,  the initial data
$$U^\pm|_{x=0}=\overline{U}_{r,l}+{\widetilde U}_0^\pm(y,z), \quad
\Psi^\pm|_{x=0}=\pm y+{\widetilde \Psi}_0^\pm(y,z)
$$
with ${\widetilde U}_0^\pm\in H^{s-{\frac 12}}(\R_+^2),
{\widetilde \Psi}_0^\pm\in H^{s+{\frac 12}}(\R_+^2)$, satisfy the compatibility
conditions of the problems \eqref{2.11} and \eqref{2.10} up to order $s-1$.
\end{quote}

Set $\widetilde U_0^{j, \pm}=\pd_x^j \widetilde U^\pm|_{x=0}, \tilde
\Psi_0^{j, \pm}=\pd_x^j\tilde \Psi^\pm|_{x=0}$ ($0\le j\le s-1$) defined
in \eqref{u-j} and \eqref{psi-j} respectively, then by the inverse trace theorem, we can construct functions $\Psi^{a, \pm},
u^{a, \pm}, w^{a, \pm}, p^{a,\pm}$ in the domain $\Omega_X$ satisfying
\bel{app}
\begin{cases}
(u^{a, \pm}-\bar u^\pm, w^{a, \pm}-\bar w^\pm, p^{a, \pm}-\bar
p)\in H^{s}(\Omega_X), \quad  \Psi^{a, \pm}\mp y\in H^{s+1}(\Omega_X),\\[2mm]
\pd_x^j(u^{a, \pm}-\bar u^\pm, w^{a, \pm}-\bar w^\pm, p^{a, \pm}-\bar
p)|_{x=0}=(\tilde u^{j, \pm}_0, \tilde w^{j, \pm}_0, \tilde p^{j,
\pm}_0), \quad \pd_x^j\tilde \Psi^{a, \pm}|_{x=0}=\tilde \Psi_0^{j,
\pm},\\[2mm]
\Psi^{a,+}=\Psi^{a,-}, \quad p^{a,+}=p^{a,-}, \qquad {\rm on}~\{y=0\}
\end{cases}
\eeq
for all $0\le j\le s-1$.

Define
\bel{app-v}
v^{a,\pm}=\pd_x\Psi^{a,\pm}u^{a, \pm}+\pd_z\Psi^{a,\pm}w^{a, \pm}.
\eeq
By using the above compatibility conditions, we know that the approximate solutions $U^a=(U^{a,+}, U^{a, -})^T$ with $U^{a,\pm}=(u^{a,\pm}, v^{a,\pm},w^{a,\pm},p^{a,\pm})^T$, and $\Psi^a=(\Psi^{a,+}, \Psi^{a, -})^T$ satisfy
\begin{equation}\label{5.8}
\begin{cases}
\pd_x^j L(U^{a,\pm}, \Psi^{a, \pm})U^{a,\pm}|_{x=0}=0, \quad \text{for}
\,\,j=0,\ldots,
s-1,\\
{\mathcal B}(U^{a,+}, U^{a,-},\psi^a)=0, \quad \text{on} \,\, \{y=0\},
\end{cases}
\end{equation}
where $\psi^a=\Psi^{a,+}|_{y=0}=\Psi^{a,-}|_{y=0}
$, $L(U^\pm, \Psi^\pm)U^\pm$ is defined in \eqref{212} and ${\mathcal B}(U^{a,+}, U^{a,-},
\psi^a)$ denotes the boundary conditions given in \eqref{2.11},
and
\begin{equation}\label{estua}
\begin{split}
  &\|\dot{U}^{a, +}\|_{H^{s}(\Omega_X)}+\|\dot{U}^{a, -}\|_{H^{s}(\Omega_X)}+\|\dot{\Psi}^{a, +}\|_{H^{s+1}(\Omega_X)}
  +\|\dot{\Psi}^{a,
  -}\|_{H^{s+1}(\Omega_X)}\\
  &\qquad +\|\psi^a\|_{H^{s+\frac{1}{2}}(\omega_X)}
  \le C\left(\|\widetilde U_0^\pm\|_{H^{s-\frac{1}{2}}(\R_+^2)}
  +\|\psi_0\|_{H^{s}(\R)}\right),
  \end{split}
\end{equation}
with
$$\dot{U}^{a,+}=U^{a, +}-\overline U_r, \quad \dot{U}^{a,-}=U^{a, -}-\overline U_l, \quad \dot{\Psi}^{a,\pm}=\Psi^{a, \pm}\mp y.$$

Moreover, if $\widetilde U_0^\pm, \psi_0$ are properly small, we have
\[
\pd_y\Psi^{a,+}\ge \frac{2}{3}, \quad \pd_y\Psi^{a,-}\le
-\frac{2}{3}, \quad \forall\, (x,y,z)\in \Omega_X.
\]

Denoting by
\begin{equation}\label{5.11}
\begin{cases}
V^\pm=U^\pm-U^{a,\pm}, \\[2mm]
\Phi^\pm=\Psi^\pm-\Psi^{a,\pm},
\end{cases}
\end{equation}
then from (\ref{5.8}), (\ref{2.11}) and (\ref{2.10}) we know that $(V^\pm, \Phi^\pm)$
satisfy the following problem:
\begin{equation}\label{5.12}
\begin{cases}
{\mathcal L}(V^\pm, \Phi^\pm)V^\pm=f_a^\pm, \quad {\rm in}~ \{x>0, ~y>0\}\\[2mm]
{\mathcal E}(V^\pm, \Phi^\pm)=0, \quad {\rm in}~ \{x>0, ~y>0\}\\[2mm]
{\B}(V^+,V^-, \phi)=0,\quad {\rm on}~ \{y=0\}\\[2mm]
V^\pm|_{x\le 0}=0, \quad \Phi^\pm|_{x\le 0}=0
\end{cases}\end{equation}
where $\phi=\Phi^+|_{y=0}=\Phi^-|_{y=0}$,
$$f_a^\pm=\begin{cases}
-L(U^{a,\pm}, \Psi^{a,\pm})U^{a,\pm}, \quad {\rm if}~x>0\\[2mm]
0, \quad {\rm if}~x\le 0
\end{cases}$$
$${\mathcal L}(V^\pm, \Phi^\pm)V^\pm=L(U^{a,\pm}+V^\pm, \Psi^{a,\pm}+\Phi^\pm)(U^{a,\pm}+V^\pm)-
L(U^{a,\pm}, \Psi^{a,\pm})U^{a,\pm},$$
$${\mathcal E}(V^\pm, \Phi^\pm)=
\partial_x(\Psi^{a,\pm}+\Phi^\pm)V_1^\pm-V_2^\pm+\partial_z(\Psi^{a,\pm}+\Phi^\pm)V_3^\pm+\partial_x\Phi^\pm u^{a,\pm}+\partial_z\Phi^\pm w^{a,\pm},$$
and
$${\B}(V^+,V^-, \phi)={\mathcal B}(U^{a,+}+V^+,U^{a,-}+V^-, \psi^a+\phi).$$

Moreover, from \eqref{estua}, we have
\begin{equation}\label{fa}
\|f^\pm_a\|_{H^{s}(\Omega_X)}\le C\left(\|\widetilde
U_0^\pm\|_{H^{s+\frac{1}{2}}(\R_+^2)}
  +\|\psi_0\|_{H^{s+1}(\R)}\right).
\end{equation}

\subsection{Description of the iteration scheme.}

To construct the Nash-Moser-H\"ormander iteration scheme for the nonlinear problem \eqref{5.12},
first let us recall a family of smoothing operators from \cite{ali, Chen-Wang1, cou2} as follows:
\begin{equation}\label{5.14}
\{S_\theta\}_{\theta>0}:~ L^2(\Omega_X)\longrightarrow \bigcap_{s\ge 0}H^s(\Omega_X)
\end{equation}
satisfying
\begin{equation}\label{5.15}
\begin{cases}
\|S_\theta u\|_{H^s(\Omega_X)}\le C\theta^{(s-\alpha)_+}\|u\|_{H^\alpha(\Omega_X)}, \quad {\rm for ~all}~s, \alpha\ge 0\cr
\|S_\theta u-u\|_{H^s(\Omega_X)}\le C\theta^{s-\alpha}\|u\|_{H^\alpha(\Omega_X)}, \quad {\rm for ~all}~0\le s\le\alpha\cr
\|\frac{d}{d\theta}S_\theta u\|_{H^s(\Omega_X)}\le C\theta^{s-\alpha-1}\|u\|_{H^\alpha(\Omega_X)}, \quad {\rm for ~all}~s,\alpha\ge 0
\end{cases}\end{equation}
and
\begin{equation}\label{5.16}
\|(S_\theta u_+-S_\theta u_-)|_{y=0}\|_{H^s(\omega_X)}\le C\theta^{(s+1-\alpha)_+}\|(u_+-u_-)|_{y=0}\|_{H^\alpha(\omega_X)},
\quad {\rm for ~all}~s, \alpha\ge 0.
\end{equation}

Similarly, one has a family of smoothing operators, still denoted by $\{S_\theta\}_{\theta>0}$ acting on $H^s(\omega_X)$, and
(\ref{5.15}) holds as well for norms of $H^s(\omega_X)$. Let $\theta_0\ge 1$, $\theta_n=\sqrt{\theta_0^2+n}$ for any $n\ge 1$, and $S_{\theta_n}$ be the associated smoothing
operators defined above.

For the problem \eqref{5.12}, let $V^{\pm}_0=\Phi^{\pm}_0\equiv 0$, and suppose that for any fixed $n\ge 0$,  the approximate solutions $\{(V^{\pm}_k, \Phi^{\pm}_k)\}_{1\le k\le n}$ of \eqref{5.12} have been constructed, satisfying
\begin{equation}\label{5.17}
\begin{cases}
(V^{\pm}_k, \Phi^{\pm}_k)|_{x\le 0}=0, \cr
\Phi^{+}_k|_{y=0}=\Phi^{-}_k|_{y=0}=\phi_k,
\end{cases}(1\le k\le n).
\end{equation}
The $(n+1)$-th approximate solutions $(V^{\pm}_{n+1}, \Phi^{\pm}_{n+1})$ of \eqref{5.12} is constructed as
\begin{equation}\label{defn1}
  V^{\pm}_{n+1}=V^{\pm}_{n}+\delta V^{\pm}_{n},
\quad \Phi^{\pm}_{n+1}=\Phi^{\pm}_{n}+\delta\Phi^{\pm}_{n}, \quad \phi_{n+1}=\phi_n+\delta\phi_n,
\end{equation}
where the increments $\delta V^{\pm}_{n}, \delta \Phi^{\pm}_n, \delta\phi_n$ satisfy the following linear problem
\begin{equation}\label{5.19}
  \begin{cases}
    L'_e(U^{a, \pm}+V^\pm_{n+\frac{1}{2}}, \Psi^{a, \pm}+S_{\theta_n}\Phi^\pm_{n})\delta \tilde V^\pm_n=f^\pm_n, \quad \text{in} ~\Omega_X,\\
    B'_{n+\frac{1}{2}}(\delta \tilde V^+_n, \delta \tilde V^-_n, \delta \phi_n)=g_n,\quad \text{on} \,\,\{y=0\}\\
    \delta \tilde V^\pm_n=0, \quad \delta \phi_n=0, \quad \text{for}\,\, x\le 0,
  \end{cases}
\end{equation}
where $L'_e(U^\pm, \Phi^\pm)V^\pm$ is the effective linearized operator defined in \eqref{lieqr},  $V^\pm_{n+\frac{1}{2}}$ is a modified state of $V^\pm_n$ such that the constraint \eqref{eikonal} holds for $(U^{a,\pm}+V^\pm_{n+\frac{1}{2}}, \Psi^{a,\pm}+S_{\theta_n}\Phi^\pm_{n})$, which will be given in \eqref{f3.73}-\eqref{f3.74},
\[
B'_{n+\frac{1}{2}}=B'_{e, (U^a+V_{n+\frac{1}{2}}, \psi^a+S_{\theta_n}\phi_{n})}
\]
is the effective boundary operator defined in \eqref{lbc} at the state $(U^a+V_{n+\frac{1}{2}}, \psi^a+S_{\theta_n}\phi_{n})$,
\bel{5.20}
\delta \tilde V^\pm_n=\delta V^\pm_n-\delta \Phi^\pm_n\frac{\pd_y(U^{a, \pm}+V^\pm_{n+\frac{1}{2}})}{\pd_y(\Psi^{a, \pm}+S_{\theta_n}\Phi^\pm_{n})}
\eeq
is the good unknown introduced in \eqref{vgu}.

To define the source term $f_n^\pm$ for the equations of \eqref{5.19}, obviously, we have
\begin{equation}\label{5.21}
\begin{split}
  &\,\,L(U^{a, \pm}+V_{n+1}^\pm, \Psi^{a,\pm}+\Phi^\pm_{n+1})(U^{a,\pm}+V^\pm_{n+1})-L(U^{a,\pm}+V^\pm_{n}, \Psi^{a,\pm}+\Phi^\pm_{n}))(U^{a,\pm}+V^\pm_{n})\\
  &=L'(U^{a,\pm}+V^\pm_n, \Psi^{a,\pm}+\Phi^\pm_n)(\delta V^\pm_n, \delta \Phi_n^\pm)+e^{\pm,1}_n\\
  &=L'(U^{a,\pm}+S_{\theta_n}V^\pm_n, \Psi^{a,\pm}+S_{\theta_n}\Phi^\pm_n)(\delta V^\pm_n, \delta \Phi_n^\pm)+e^{\pm,1}_n+e^{\pm,2}_n\\
  &=L'(U^{a,\pm}+V^\pm_{n+\frac{1}{2}}, \Psi^{a,\pm}+S_{\theta_n}\Phi^\pm_n)(\delta V^\pm_n, \delta \Phi_n^\pm)+e^{\pm,1}_n+e^{\pm,2}_n+e^{\pm,3}_n\\
  &=L_e'(U^{a,\pm}+V^\pm_{n+\frac{1}{2}}, \Psi^{a,\pm}+S_{\theta_n}\Phi^\pm_n)\delta \tilde  V^\pm_n+e^{\pm,1}_n+e^{\pm,2}_n+e^{\pm,3}_n+e^{\pm,4}_n\\
  \end{split}
\end{equation}
with errors $e^{\pm,1}_n$ arising from the Newton iteration, $e^{\pm,2}_n$ and $e^{\pm,3}_n$ arising from the substitutions in the coefficient functions of the linearized operator $L'$ from $V_n^\pm$ to $S_{\theta_n}V_n^\pm$, and from $S_{\theta_n}V_n^\pm$ to $V_{n+{\frac 12}}^\pm$ respectively, and $$e^{\pm,4}_n=\frac{\delta\Phi_n^\pm}{\pd_y(\Psi^{a,\pm}+S_{\theta_n}\Phi^\pm_n)}\pd_y
\left[L(U^{a,\pm}+V^\pm_{n+\frac{1}{2}}, \Psi^{a, \pm}+S_{\theta_n}\Phi^\pm_n)(U^{a,\pm}+V^\pm_{n+\frac{1}{2}})\right]$$
arising from the use of good unknown $\delta\tilde{V}_n^\pm$ from $\delta {V}_n^\pm$. To guarantee the limit of $(V^\pm_n, \Phi^\pm_n)$ defined in \eqref{defn1}-\eqref{5.19} being the solution of the problem \eqref{5.12}, we define the source term $f^\pm_n$  to satisfy,
\[
\sum\limits_{j=0}\limits^n f^\pm_j+S_{\theta_n}E^\pm_n=S_{\theta_n}f^\pm_a, \qquad
\forall n\ge 1
\] where $f_0^\pm=S_{\theta_0}f^\pm_a$, and $E^\pm_n=\sum\limits_{l=0}\limits^{n-1}e_l^{\pm}$ with $e^\pm_{l}=\sum\limits_{j=1}\limits^4 e_l^{\pm,j}$,  i.e.
\bel{f-n}
f_n^\pm=(S_{\theta_n}-S_{\theta_{n-1}})f_a^\pm-(S_{\theta_n}-S_{\theta_{n-1}})E^\pm_{n-1}
-S_{\theta_{n}}e^\pm_{ n-1}.
\eeq

The source term $g_n$ of the boundary condition given in \eqref{5.19} can be defined in a similar way.
It is obvious that
\begin{equation}\label{5.23}
\begin{array}{ll}
&{\mathcal B}(U^{a,+}+V^+_{n+1}, U^{a,-}+V^{-}_{n+1}, \psi^a+\phi_{n+1})-
{\mathcal B}(U^{a,+}+V^+_{n}, U^{a,-}+V^{-}_{n}, \psi^a+\phi_{n})\\
&\hspace{.2in}={\mathcal B}'_{(U^{a,\pm}+V^{\pm}_n, \psi^a+\phi_{n})}\cdot (\delta {V^{+}_n},\delta {V^{-}_n}, \delta {\phi_{n}})+\tilde{e}_{ n}^{(1)}\\[2mm]
&\hspace{.2in}={\mathcal B}'_{(U^{a,\pm}+S_{\theta_n}V^{\pm}_n, \psi^a+S_{\theta_n}\phi_{n})}\cdot (\delta {V^{+}_n},\delta {V^{-}_n}, \delta {\phi_{n}})+\tilde{e}_{ n}^{(1)}+\tilde{e}_{ n}^{(2)}\\[2mm]
&\hspace{.2in}={\mathcal B}'_{(U^{a,\pm}+V^{\pm}_{n+{\frac 12}}, \psi^a+S_{\theta_n}\phi_n)}\cdot (\delta {V^{+}_n},\delta {V^{-}_n}, \delta {\phi_{n}})+\tilde{e}_{ n}^{(1)}+\tilde{e}_{ n}^{(2)}+\tilde{e}_{ n}^{(3)}\\[2mm]
&\hspace{.2in}=
B'_{n+\frac{1}{2}}(\delta {\tilde V^{+}_n},\delta {\tilde V^{-}_n}, \delta {\phi_{n}})
+\tilde{e}_{ n}^{(1)}+\tilde{e}_{ n}^{(2)}+\tilde{e}_{ n}^{(3)}+\tilde{e}_{ n}^{(4)}\\[2mm]
\end{array}
\end{equation}
with the components of $\tilde{e}_{n}^{(4)}$ being
$$\tilde{e}_{n,1}^{(4)}=\frac{\delta\phi_n}{\partial_y(\Psi^{a,+}+S_{\theta_n}\Phi_n^+)}
\left(
(\psi^a+S_{\theta_n}\phi_n)_x(U_1^{a,+}+V_{n+{\frac 12},1}^+)_y
-(U_2^{a,+}+V_{n+{\frac 12},2}^+)_y+
(\psi^a+S_{\theta_n}\phi_n)_z(U_3^{a,+}+V_{n+{\frac 12},3}^+)_y
\right),$$
$$\tilde{e}_{n,2}^{(4)}=\frac{\delta\phi_n}{\partial_y(\Psi^{a,-}+S_{\theta_n}\Phi_n^-)}
\left(
(\psi^a+S_{\theta_n}\phi_n)_x(U_1^{a,-}+V_{n+{\frac 12},1}^-)_y
-(U_2^{a,-}+V_{n+{\frac 12},2}^-)_y+
(\psi^a+S_{\theta_n}\phi_n)_z(U_3^{a,-}+V_{n+{\frac 12},3}^-)_y
\right)$$
and
$$\tilde{e}_{n,3}^{(4)}=\delta\phi_n\left(
\frac{\partial_y(U_4^{a,+}+V_{n+{\frac 12},4}^+)}{\partial_y(\Psi^{a,+}+S_{\theta_n}\Phi_n^+)}-\frac{\partial_y(U_4^{a,-}+V_{n+{\frac 12},4}^-)}{\partial_y(\Psi^{a,-}+S_{\theta_n}\Phi_n^-)}
\right)
$$

Noting that
$
{\mathcal B}(U^{a,+}, U^{a,-}, \psi^a)=0
$, to guarantee the limit of $(V^\pm_n, \Phi^\pm_n)$ satisfies the boundary condition given in \eqref{5.12}, we define the source term $g_n$  given in \eqref{5.19} to satisfy,
\begin{equation}\label{g-n}
\sum_{j=0}^ng_j+S_{\theta_n}(\sum_{j=0}^{n-1}\tilde{e}_{j})=0
\end{equation}
by induction on $n$, with $g_0=0$, and
$\tilde{e}_{n}=\sum_{j=1}^4\tilde{e}_{n}^{(j)}
$.

The next goal is to construct $\delta\Phi^{\pm}_n$ such that $\delta\Phi^{\pm}_n|_{y=0}=\delta\phi^{n}$, this will use the idea from \cite{cou2}.

From the first two components of the boundary conditions given in (\ref{5.19}), we know that $\delta\phi^n$ satisfies
\begin{equation}\label{5.25}
\begin{array}{l}
(U^{a,+}_1+V^+_{n+\frac 12, 1})\partial_x(\delta\phi_n)+(U^{a,+}_3+V^+_{n+\frac 12, 3})\partial_z(\delta\phi_n)\\
\hspace{.3in}-\delta\tilde{V}^+_{n,2}+\partial_{x}(\psi^a+S_{\theta_n}\phi_n)\delta\tilde{V}^+_{n,1}
+\partial_{z}(\psi^a+S_{\theta_n}\phi_n)\delta\tilde{V}^+_{n,3}=
g_{n,1}
\end{array}
\end{equation}
and
\begin{equation}\label{5.26}
\begin{array}{l}
(U^{a,-}_1+V^-_{n+\frac 12, 1})\partial_x(\delta\phi_n)+(U^{a,-}_3+V^-_{n+\frac 12, 3})\partial_z(\delta\phi_n)\\
\hspace{.3in}-\delta\tilde{V}^-_{n,2}+\partial_{x}(\psi^a+S_{\theta_n}\phi_n)\delta\tilde{V}^-_{n,1}
+\partial_{z}(\psi^a+S_{\theta_n}\phi_n)\delta\tilde{V}^-_{n,3}=
g_{n,2}
\end{array}
\end{equation}
on $\{y=0\}$, this inspires us to define $\delta\Phi^{\pm}_n$ by solving the problems
\begin{equation}\label{5.27}
\begin{cases}
(U^{a,+}_1+V^+_{n+\frac 12, 1})\partial_x(\delta\Phi^+_n)+(U^{a,+}_3+V^+_{n+\frac 12, 3})\partial_z(\delta\Phi^+_n)\\
\hspace{.3in}-\delta\tilde{V}^+_{n,2}+\partial_{x}(\Psi^{a,+}+S_{\theta_n}\Phi_n^+)\delta\tilde{V}^+_{n,1}
+\partial_{z}(\Psi^{a,+}+S_{\theta_n}\Phi_n^+)\delta\tilde{V}^+_{n,3}=
{\mathbb E}g_{n,1}+h_n^+\\
\delta\Phi^{+}_n|_{x\le 0}=0
\end{cases}
\end{equation}
and
\begin{equation}\label{5.28}
\begin{cases}
(U^{a,-}_1+V^-_{n+\frac 12, 1})\partial_x(\delta\Phi^-_n)+(U^{a,-}_3+V^-_{n+\frac 12, 3})\partial_z(\delta\Phi^-_n)\\
\hspace{.3in}-\delta\tilde{V}^-_{n,2}+\partial_{x}(\Psi^{a,-}+S_{\theta_n}\Phi^-_n)\delta\tilde{V}^-_{n,1}
+\partial_{z}(\Psi^{a,-}+S_{\theta_n}\Phi_n^-)\delta\tilde{V}^-_{n,3}=
{\mathbb E}g_{n,2}+h_n^-\\
\delta\Phi^{-}_n|_{x\le 0}=0
\end{cases}
\end{equation}
where ${\mathbb E}$ is a proper extension operator from $H^s(\omega_X)$ to $H^{s+{\frac 12}}(\Omega_X)$, and $h^\pm_n$
need to be determined such that $h_n^\pm|_{x\le 0}=h_n^\pm|_{y=0}=0$, and $\delta\Phi^+_n=\delta\Phi^-_n$ on $\{y=0\}$.

To determine $h_n^\pm$, let us study an iteration scheme for the eikonal equation ${\mathcal E}(V^\pm, \Phi^\pm)=0$ given in (\ref{5.12}).

Obviously, we have
\begin{equation}\label{5.29}
{\mathcal E}(V^{\pm}_{n+1}, \Phi^{\pm}_{n+1})-{\mathcal E}(V^{\pm}_n,\Phi^{\pm}_n)=
{\mathcal E}'_{(V^{\pm}_{ n+\frac 12}, S_{\theta_n}\Phi_n^\pm)}(\delta \tilde{V}^{\pm}_n,
\delta \Phi^{\pm}_n)+\overline{e}_{\pm,n}
\end{equation}
where
\begin{equation}\label{5.30}
\begin{array}{ll}
{\mathcal E}'_{(V^{\pm}, \Phi^{\pm})}(W^{\pm},
\Theta^{\pm})=&
(u^{a,\pm}+V_1^\pm)\partial_x\Theta^\pm
+(w^{a,\pm}+V_3^\pm)\partial_z\Theta^\pm\\[2mm]
& -W^\pm_2+\partial_{x}(\Psi^{a,\pm}+\Phi^{\pm})
W_1^{\pm}
+\partial_{z}(\Psi^{a,\pm}+\Phi^{\pm})W_3^{\pm}
\end{array}
\end{equation}
is the linearized operator of $\mathcal E$,
\begin{equation}\label{5.31}
\overline{e}_{\pm,n}=\sum_{j=1}^4\overline{e}_{\pm, n}^{(j)}
\end{equation}
with
\begin{equation}\label{5.32}
\begin{cases}
\overline{e}_{\pm, n}^{(1)}=
{\mathcal E}(V^{\pm}_{n+1}, \Phi^{\pm}_{n+1})-{\mathcal E}(V^{\pm}_n, \Phi^{\pm}_n)
-{\mathcal E}'_{(V^{\pm}_n, \Phi^{\pm}_n)}(\delta {V}^{\pm}_n,
\delta \Phi^{\pm}_n)\\[2mm]
\hspace{.25in}=\partial_x(\delta \Phi^\pm_n)\delta V_{n,1}^\pm+\partial_z(\delta \Phi^\pm_n)\delta V_{n,3}^\pm\\[2mm]
\overline{e}_{\pm,n}^{(2)}={\mathcal E}'_{(V^{\pm}_n, \Phi^{\pm}_n)}\cdot (\delta {V^{\pm}_n},
 \delta {\Phi^{\pm}_n})
-{\mathcal E}'_{(S_{\theta_n}V^{\pm}_n, S_{\theta_n}\Phi^{\pm}_n)}\cdot (\delta {V^{\pm}_n},
 \delta {\Phi^{\pm}_n})\\[2mm]
\hspace{.25in}=\partial_x(\delta \Phi^\pm_n)(1-S_{\theta_n})V^{\pm}_{n,1}+\partial_z(\delta \Phi^\pm_n)(1-S_{\theta_n})V^{\pm}_{n,3}\\[2mm]
\hspace{.35in}+\partial_x((1-S_{\theta_n})\delta \Phi^\pm_n)\delta V^\pm_{n,1}+\partial_z((1-S_{\theta_n})\delta \Phi^\pm_n)\delta V^\pm_{n,3}
\\[2mm]
\overline{e}_{\pm,n}^{(3)}=
{\mathcal E}'_{(S_{\theta_n}V^{\pm}_n, S_{\theta_n}\Phi^{\pm}_n)}\cdot (\delta {V^{\pm}_n},
 \delta {\Phi^{\pm}_n})
-{\mathcal E}'_{(V^{\pm}_{n+\frac 12}, S_{\theta_n}\Phi^{\pm}_n)}\cdot (\delta {V^{\pm}_n},
 \delta {\Phi^{\pm}_n})\\[2mm]
\hspace{.25in}=\partial_x(\delta \Phi^\pm_n)(S_{\theta_n}V^{\pm}_{n,1}-V^\pm_{n+{\frac 12}, 1})+\partial_z(\delta \Phi^\pm_n)(S_{\theta_n}V^{\pm}_{n,3}-V^\pm_{n+{\frac 12}, 3})\\[2mm]
\overline{e}_{\pm,n}^{(4)}=
{\mathcal E}'_{(V^{\pm}_{n+\frac 12}, S_{\theta_n}\Phi^{\pm}_n)}\cdot (\delta {V^{\pm}_n},
 \delta {\Phi^{\pm}_n})
-{\mathcal E}'_{(V^{\pm}_{n+\frac 12}, S_{\theta_n}\Phi^{\pm}_n)}\cdot (\delta {\tilde{V}^{\pm}_n},
 \delta {\Phi^{\pm}_n})\\[2mm]
\hspace{.25in}=\frac{\delta\Phi_n^\pm}{\partial_y(\Psi^{a,\pm}+
S_{\theta_n}\Phi^{\pm}_n)}
\left(
(\Psi^{a,\pm}+S_{\theta_n}\Phi^{\pm}_n)_x(U_1^{a,\pm}+V_{n+{\frac 12},1}^\pm)_y
-(U_2^{a,\pm}+V_{n+{\frac 12},2}^\pm)_y+
(\Psi^{a,\pm}+S_{\theta_n}\Phi^{\pm}_n)_z(U_3^{a,\pm}+V_{n+{\frac 12},3}^\pm)_y
\right).
\end{cases}
\end{equation}

Thus, from \eqref{5.27}, \eqref{5.28}, \eqref{5.29} and ${\mathcal E}(V^{\pm}_0, \Phi^{\pm}_0)=0$ we get
\begin{equation}\label{5.33}
\begin{array}{ll}
{\mathcal E}(V^+_{n+1}, \Phi^+_{n+1}) & =
\sum\limits_{k=0}^n({\mathbb E}g_{k,1}+h_k^++\overline{e}_{+,k})\\[2mm]
&={\mathbb E}\left(({\mathcal B}(U^{a,+}+V^+_{n+1}, U^{a,-}+V^{-}_{n+1}, \psi^a+\phi_{n+1}))_1-\sum\limits_{k=0}^n\tilde{e}_{k,1}\right)+
\sum\limits_{k=0}^n(h_k^++\overline{e}_{+,k})
\end{array}
\end{equation}
by using \eqref{5.23}. Therefore, we define $h_n^+$ through
\begin{equation}\label{5.34}
\sum_{k=0}^nh_k^++S_{\theta_n}\left(\sum_{k=0}^{n-1}(\overline{e}_{+, k}-
{\mathbb E}(\tilde{e}_{k, 1}))\right)=0
\end{equation}
by induction on $k$. Similarly, from the equations of ${\mathcal E}(V_{n+1}^-, \Phi^-_{n+1})$ and $({\mathcal B}(U^{a,+}+V^+_{n+1}, U^{a,-}+V^{-}_{n+1}, \psi^a+\phi_{n+1}))_2$ given in \eqref{5.29} and \eqref{5.23} respectively, we define
$h_n^-$ by
\begin{equation}\label{5.35}
\sum_{k=0}^nh_k^-+S_{\theta_n}\left(\sum_{k=0}^{n-1}(\overline{e}_{-, k}-
{\mathbb E}(\tilde{e}_{k, 2}))\right)=0.
\end{equation}

The steps for determining $(\delta\tilde{V}^{\pm}_n, \delta\Phi^{\pm}_n, \delta\phi_n)$ are to solve
$\delta\tilde{V}^{\pm}_n$ from (\ref{5.19}) first, then to solve $\delta\Phi^{\pm}_n$ from (\ref{5.27})-(\ref{5.28}),
which yields $\delta\phi_n=\delta\Phi^{\pm}_n|_{y=0}$ satisfying (\ref{5.25}) and (\ref{5.26}).

\section{Estimate of approximate solutions and convergence}
\setcounter{equation}{0}

\subsection{Convergence of the iteration scheme}

For fixed $s_0> \frac 52$, $\alpha\ge s_0+5$ and $\alpha+6\le s_1\le 2\alpha-s_0+1$.

Suppose that the first approximate solutions constructed in \S4.1 satisfy
\begin{equation}\label{6.1}
\begin{cases}
\|\dot{U}^{a,\pm}\|_{s_1+2, X}+\|\dot{\Psi}^{a,\pm}\|_{s_1+3, X}
+\|f_a^\pm\|_{s_1+1, X}\le \delta, \\
\|f_a^\pm\|_{\alpha+1, X}/\delta \quad {\rm is~small}, \quad
\|f_a^\pm\|_{\alpha+2, X}/\delta \quad {\rm is~bounded}
\end{cases}
\end{equation}
for a small $\delta>0$, where and hereafter we shall use $\|\cdot\|_{s,X}$ to denote the norm in the space $H^s(\Omega_X)$ for simplicity.

For the iteration scheme (\ref{defn1})(\ref{5.19}), we make the following inductive
assumption
$$\begin{cases}
\|\delta V^{\pm}_k, \delta \Phi^{\pm}_k\|_{s,X}+\|\delta \phi_{k}\|_{H^{s+1}(\omega_X)}
\le\delta \theta_k^{s-\alpha-1}\Delta_k, \quad 0\le k\le n-1, s_0\le s\le s_1\cr
\|{\mathcal L}(V^{\pm}_k, \Phi^{\pm}_k)V^{\pm}_k-f_a^\pm\|_{s,X}
\le \delta \theta_k^{s-\alpha-1}, \quad 0\le k\le n, s_0\le s\le s_1-2\cr
\|{\mathbb B}(V^{+}_k, V^{-}_k, \phi_{k})\|_{H^{s-1}(\omega_X)}
\le \delta \theta_k^{s-\alpha-1}, \quad 0\le k\le n, s_0\le s\le s_1-2
\end{cases}\leqno(H_{n})$$
with $\Delta_k=\theta_{k+1}-\theta_k$.

Temporarily, we suppose the above inductive assumption being true for all $n\ge 1$, then we can conclude the main result, Theorem 2.1 immediately.

\vspace{.1in}
{\bf\sc Proof of Theorem 2.1:}

From ($H_n$) for any $n\ge 0$, we get
\begin{equation}\label{6.2}
\sum_{k\ge 0}\left(\|\delta V^{\pm}_k, \delta\Phi^{\pm}_k\|_{\alpha-1,X}+\|\delta\phi_k\|_{H^{\alpha}(\omega_X)}\right)<+\infty
\end{equation}
which implies that there exist $V^\pm, \Phi^\pm$ in $H^{\alpha-1}(\Omega_X)$ and $\phi$ in $H^{\alpha}(\omega_X)$
such that
\begin{equation}\label{6.3}
\begin{cases}
(V^{\pm}_n, \Phi^{\pm}_n)\longrightarrow (V^{\pm}, \Phi^{\pm})\quad {\rm in}\quad
H^{\alpha-1}(\Omega_X)\cr
\phi_n\longrightarrow \phi\quad {\rm in}\quad
H^{\alpha}(\omega_X)
\end{cases} \quad {\rm as}~ n\to +\infty,
\end{equation}
and $(V^\pm, \Phi^\pm, \phi)$ are solutions to the problem \eqref{5.12}.

Thus, we conclude

\vspace{.1in}
\begin{theorem} 
For any fixed $\alpha> \frac{15}{2}$ and $s_1\ge \alpha +6$. Suppose that $\psi_0\in H^{s_1}(\R)$,
$U_0^\pm-\overline{U}_{r,l}\in H^{s_1-\frac 12}(\R^2_+)$ satisfy the compatibility conditions of the problem
(\ref{2.11}) up to order $s_1-1$, and the conditions \eqref{2.15} and (\ref{6.1}) are satisfied. Then, there exist solutions
$(V^\pm, \Phi^\pm)\in H^{\alpha-1}(\Omega_X)$ and $\phi\in H^{\alpha}(\omega_X)$ to the problem (\ref{5.12}).
\end{theorem}

\vspace{.1in}
The remaining main task is to estimate solutions of problems (\ref{5.19})
and (\ref{5.27})-(\ref{5.28}) to verify the inductive assumption ($H_n$) for all $n\ge 1$.

\subsection{Estimates of errors and approximate solutions}

The main step for verifying ($H_{n+1}$) under the assumption of ($H_{n}$) is to estimate errors appeared in the Nash-Moser iteration scheme (\ref{5.19})
and (\ref{5.27})-(\ref{5.28}), we shall mainly fellow the arguments similar to that given in \cite{Chen-Wang, cou2}.

First, from  ($H_{n}$), we immediately have

\begin{lemma}\label{lemma6.1}
The following estimates hold:
\begin{equation}\label{f3.51}
\begin{cases}
\|V^{\pm}_k, \Phi^{\pm}_k\|_{s,X}+\|\phi_{k}\|_{H^{s+1}(\omega_X)}
\le C\delta \theta_k^{(s-\alpha)_+}, \quad s_0\le s\le s_1, ~s\neq \alpha\cr
\|V^{\pm}_k, \Phi^{\pm}_k\|_{\alpha,X}+\|\phi_{k}\|_{H^{\alpha+1}(\omega_X)}
\le C\delta \log\theta_k\cr
\|S_{\theta_k}V^{\pm}_k, S_{\theta_k}\Phi^{\pm}_k\|_{s,X}
+\|S_{\theta_k}\phi_{k}\|_{H^{s+1}(\omega_X)}
\le C\delta \theta_k^{(s-\alpha)_+}, \quad s\ge s_0, s\neq \alpha\cr
\|S_{\theta_k}V^{\pm}_k, S_{\theta_k}\Phi^{\pm}_k\|_{\alpha,X}
+\|S_{\theta_k}\phi_{k}\|_{H^{\alpha+1}(\omega_X)}
\le C\delta \log\theta_k\cr
\|(I-S_{\theta_k})V^{\pm}_k, (I-S_{\theta_k})\Phi^{\pm}_k\|_{s,X}
+\|(I-S_{\theta_k})\phi_{k}\|_{H^{s+1}(\omega_X)}
\le C\delta \theta_k^{s-\alpha}, \quad s_0\le s\le s_1
\end{cases}
\end{equation}
for all $0\le k\le n$.
\end{lemma}

\begin{lemma}\label{lemma6.2}
 For the quadratic errors $e^{\pm, 1}_k$, $\bar{e}^{(1)}_{\pm,k}$ and
$\tilde{e}^{(1)}_k$ given in \eqref{5.21},  \eqref{5.32} and \eqref{5.23} respectively, we have
\begin{equation}\label{f3.52}
\begin{cases}
\|e^{\pm,1}_k\|_{s,X}\le C\delta^2\theta_k^{L_1(s)}\Delta_k,  \quad s_0-1\le s\le s_1-1\cr
\|\overline{e}_{\pm,k}^{(1)}\|_
{s,X}\le C\delta^2\theta_k^{s+s_0-2\alpha-2}\Delta_k,  \quad s_0\le s\le s_1-1\cr
\|\tilde{e}_{k}^{(1)}\|_{H^s(\omega_X)}\le C\delta^2\theta_k^{s+s_0-2\alpha-\frac 52}\Delta_k,  \quad s_0\le s\le s_1-\frac 12\end{cases}
\end{equation}
for all $k\le n-1$, where
\begin{equation}\label{f3.53}
L_1(s)=\max((s+1-\alpha)_++2(s_0-\alpha-1), s+s_0-2\alpha-2).
\end{equation}
\end{lemma}

\begin{proof}
We can get estimates of   $\overline{e}_{\pm,k}^{(1)}$ and $\tilde{e}_{k}^{(1)}$ 
much easier than that of $e^{\pm,1}_k$ by using their explicit expressions and the inductive assumption ($H_{n}$),
so we shall only study $e^{\pm,1}_k$ in detail. Obviously, we have
\begin{equation}\label{f3.54}
e^{\pm,1}_k=\int_0^1(1-\tau)L{''}_{(U^{a,\pm}+V^{\pm}_k+\tau\delta V^{\pm}_k;
\Psi^{a,\pm}+\Phi^{\pm}_k+\tau\delta \Phi^{\pm}_k)}((\delta V^{\pm}_k,\delta \Phi^{\pm}_k),
(\delta V^{\pm}_k,\delta \Phi^{\pm}_k))d\tau
\end{equation}

From (\ref{6.1}), ($H_{n}$) and Lemma \ref{lemma6.1}, we get
\begin{equation}\label{6.24}
\begin{cases}
\|\dot{U}^{a,\pm}+V^{\pm}_k+\tau\delta V^{\pm}_k\|_{s,X}\le C\delta
(1+\theta_k^{(s-\alpha)_+}+\theta_k^{s-\alpha-2}), \qquad \forall s_0\le s\le s_1, \quad s\neq \alpha\\[2mm]
\|\dot{U}^{a,\pm}+V^{\pm}_k+\tau\delta V^{\pm}_k\|_{\alpha,X}\le C\delta
(1+\log \theta_k+\theta_k^{-2})
\end{cases}
\end{equation}
for all $k\le n-1$ and $0\le \tau\le 1$, which implies
\begin{equation}\label{f3.55}
\sup_{0\le \tau\le 1}\|\dot{U}^{a,\pm}+V^{\pm}_k+\tau\delta V^{\pm}_k\|_
{W^{1,\infty}(\Omega_X)}\le C\delta.
\end{equation}

On the other hand, obviously we have
\begin{equation}\label{f3.56}
\begin{array}{l}
\|L''_{(U^\pm, \Psi^\pm)}((V^{\pm}_1, \Phi^{\pm}_1),(V^{\pm}_2, \Phi^{\pm}_2))\|_{s,X}
\le C(\|\dot{U}^\pm, \dot{\Psi}^\pm\|_{s+1,X}\|V^{\pm}_1, \Phi^{\pm}_1\|_{W^{1,\infty}}
\|V^{\pm}_2, \Phi^{\pm}_2\|_{W^{1,\infty}}\\[2mm]
\hspace{.3in}+
\|V^{\pm}_1, \Phi^{\pm}_1\|_{s+1,X}
\|V^{\pm}_2, \Phi^{\pm}_2\|_{W^{1,\infty}}+
\|V^{\pm}_1, \Phi^{\pm}_1\|_{W^{1,\infty}}
\|V^{\pm}_2, \Phi^{\pm}_2\|_{s+1,X}).
\end{array}\end{equation}

Therefore, by using \eqref{6.24}, ($H_{n}$) and Lemma \ref{lemma6.1}, we have
\begin{equation}
\label{6.28}
\begin{array}{ll}
\|e^{\pm, 1}_k\|_{s,X}
& \le C\delta(\delta\theta_k^{s_0-\alpha-1}\Delta_k)^2
(1+\theta_k^{(s+1-\alpha)_+}+\theta_k^{s-\alpha-1})+C\delta^2\theta_k^{s+s_0-
2\alpha-1}
\Delta_k^2\\[2mm]
& \le C\delta^2\theta_k^{L_1(s)}\Delta_k
\end{array}
\end{equation}
as $s_0>\frac 52$, where $L_1(s)=\max((s+1-\alpha)_++2(s_0-\alpha-1), s+s_0-2\alpha-2)$, for all $s_0-1\le s\le s_1-1$ with $s\neq \alpha-1$, and
\begin{equation}
\label{6.29}
\|e^{\pm, 1}_k\|_{\alpha-1,X}
\le C\delta(\delta\theta_k^{s_0-\alpha-1}\Delta_k)^2
(1+\log\theta_k+\theta_k^{-2})+C\delta^2\theta_k^{s_0-\alpha-2}
\Delta_k^2\le C\delta^2\theta_k^{L_1(\alpha-1)}\Delta_k.
\end{equation}

Thus, we conclude the first result given in (\ref{f3.52}).\end{proof}

\begin{lemma}\label{lemma6.3}
For the errors $e^{\pm, 2}_k$, $\bar{e}^{(2)}_{\pm,k}$ and
$\tilde{e}^{(2)}_k$ given in \eqref{5.21},  \eqref{5.32} and \eqref{5.23} respectively, we have
\begin{equation}\label{f3.62}
\begin{cases}
\|e_{k}^{\pm, 2}\|_{s,X}\le C\delta^2\theta_k^{L_2(s)}\Delta_k, \quad s_0-1\le s\le s_1-1\cr
\|\overline{e}_{\pm,k}^{(2)}\|_
{s,X}\le C\delta^2\theta_k^{s+s_0-2\alpha}\Delta_k, \quad s_0\le s\le s_1-1\cr
\|\tilde{e}_{k}^{(2)}\|_{H^s(\omega_X)}\le C\delta^2\theta_k^{s+s_0-2\alpha-\frac 12}\Delta_k, \quad s_0\le s\le s_1-\frac 12
\end{cases}
\end{equation}
for all $k\le n-1$, where
\begin{equation}\label{f3.63}
L_2(s)=\max((s+1-\alpha)_++2(s_0-\alpha),  s+s_0-2\alpha).
\end{equation}
\end{lemma}

\begin{proof} As in Lemma \ref{lemma6.2}, we shall only study $e^{\pm,2}_{k}$ in detail, the estimate of   $\overline{e}_{\pm,k}^{(2)}$ and $\tilde{e}_{k}^{(2)}$ can be easily obtained  by using  ($H_{n}$).

 From the definition of $e^{\pm,2}_{k}$, obviously we have
\begin{equation}\label{f3.64}
\begin{array}{l}
e^{\pm,2}_{k}=\int_0^1L{''}_{(U^{a,\pm}+S_{\theta_k}V^{\pm}_k+\tau(1-S_{\theta_k})V^{\pm}_k;
\Psi^{a,\pm}+S_{\theta_k}\Phi^{\pm}_k+\tau(1-S_{\theta_k})\Phi^{\pm}_k)}
((\delta V^{\pm}_k,\delta \Phi^{\pm}_k),((1-S_{\theta_k})V^{\pm}_k,(1-S_{\theta_k})\Phi^{\pm}_k))d\tau
\end{array}
\end{equation}

As in (\ref{f3.55}), from the assumption ($H_{n}$) we have
\begin{equation}\label{f3.65}
\sup_{0\le \tau\le 1}\left(
\|\dot{U}^{a,\pm}+S_{\theta_k}V^{\pm}_k+\tau(1-S_{\theta_k})V^{\pm}_k\|_
{W^{1,\infty}(\Omega_T)}+\|\dot{\Psi}^{a,\pm}+S_{\theta_k}\Phi^{\pm}_k+\tau(1-S_{\theta_k})\Phi^{\pm}_k\|_
{W^{1,\infty}(\Omega_T)}\right)\le C\delta.
\end{equation}

Therefore, by using (\ref{f3.56}) in
(\ref{f3.64}) we obtain
\begin{equation}\label{f3.67}
\begin{array}{l}
\|e^{\pm,2}_k\|_{s,X}
\le C\left(\|\delta V^{\pm}_k,\delta \Phi^{\pm}_k\|_{W^{1,\infty}}
\|(1-S_{\theta_k})( V^{\pm}_k,\Phi^{\pm}_k)\|_{W^{1,\infty}}
(\|\dot{U}^{a,\pm}, \dot{\Psi}^{a,\pm}\|_{s+1, X}\right.\\[2mm]
\hspace{.7in}
+\|S_{\theta_k}(V^{\pm}_k,\Phi^{\pm}_k)\|_{s+1,X}+
\|(1-S_{\theta_k})(V^{\pm}_k,\Phi^{\pm}_k)\|_{s+1,X})\\[2mm]
\hspace{.7in}
+\|\delta V^{\pm}_k,\delta \Phi^{\pm}_k\|_{s+1,X}
\|(1-S_{\theta_k})(V^{\pm}_k,\Phi^{\pm}_k)\|_{W^{1,\infty}}\\[2mm]
\hspace{.7in}
\left.+\|\delta V^{\pm}_k,\delta \Phi^{\pm}_k\|_{W^{1,\infty}}
\|(1-S_{\theta_k})(V^{\pm}_k,\Phi^{\pm}_k)\|_{s+1, X}\right)
\end{array}
\end{equation}

By using the properties of smoothing operators, the assumption ($H_{n}$) and Lemma \ref{lemma6.1} in
(\ref{f3.67}) we conclude the first estimate given in (\ref{f3.62}) when $s_0-1\le s\le s_1-1$.
\end{proof}

\vspace{.1in}
To estimate the error $e^{\pm, 3}_k$, let us define the modified state $V^{\pm}_{n+\frac 12}$
first, this will be done in an idea similar to that given in \cite{cou2, Chen-Wang, tra1}.

To guarantee that the boundary $\{y=0\}$ is uniformly characteristic at each step iteration (\ref{5.19}),
we require that
\begin{equation}\label{f3.72}
({\mathbb B}(V^{+}_{n+\frac 12},V^{-}_{n+\frac 12},
S_{\theta_n}\phi_{n}))_i^\pm=0, \quad\text{on}~ \{y=0\}
\end{equation}
for $i=1,2$ and all $n\in \N$, which leads to define
\begin{equation}\label{f3.73}
V^{\pm}_{n+\frac 12,j}=S_{\theta_n}V^{\pm}_{n,j}, \quad j\in\{1, 3\}
\end{equation}
and
\begin{equation}\label{f3.74}
\begin{array}{ll}
V^{\pm}_{n+\frac 12, 2}=&
\partial_{x}(\Psi^{a,\pm}+S_{\theta_n}\Phi^{\pm}_n)V^{\pm}_{n+\frac 12,1}
+\partial_{z}(\Psi^{a,\pm}+S_{\theta_n}\Phi^{\pm}_n)V^{\pm}_{n+\frac 12,3}\\[2mm]
&
+u^{a,\pm}\partial_{x}(S_{\theta_n}\Phi^{\pm}_n)+w^{a,\pm}\partial_{z}(S_{\theta_n}\Phi^{\pm}_n)
\end{array}
\end{equation}

\begin{lemma}\label{lemma6.4}
For the modified state $V^{\pm}_{n+\frac 12}$ defined at above, we have
\begin{equation}\label{f3.75}
\|V^{\pm}_{n+\frac 12}-S_{\theta_k}V^{\pm}_{n}\|_{s, X}\le C\delta \theta_n^{s+1-\alpha}
\end{equation}
for any $s_0\le s\le s_1+3$.\end{lemma}

We can prove this lemma in the same way as given in \cite[\S7.4]{cou2}, so we omit it here.

\vspace{.1in} From the definition of the intermediate state $V_{n+\frac 12}^\pm$ given in \eqref{f3.73}-\eqref{f3.74}, we know
\begin{equation}\label{6.46}
\overline{e}^{(3)}_{\pm,n}=\tilde{e}^{(3)}_n\equiv 0
\end{equation}
for these two errors given in \eqref{5.32} and \eqref{5.23} respectively. The representation of the error $e^{\pm, 3}_k$ given in \eqref{5.21}
is similar to that of
$e^{\pm, 2}_k$, so by using Lemma \ref{lemma6.4} and the same argument as the proof of Lemma \ref{lemma6.3}, we conclude

\begin{lemma}\label{lemma6.5}
 For the error $e^{\pm, 3}_k$, we have
\begin{equation}\label{f3.83}
\|e^{\pm, 3}_k\|_{s,X}\le C\delta^2\theta_k^{L_3(s)}\Delta_k, \quad s_0-1\le s\le s_1-1
\end{equation}
for all $k\le n-1$, where
\begin{equation}\label{f3.84}
L_3(s)=\max((s+1-\alpha)_++2(s_0-\alpha), s+2s_0-3\alpha+2, s+s_0-2\alpha+1, 2(s_0-\alpha)+1 ).
\end{equation}\end{lemma}

\begin{lemma}\label{lemma6.6}
For the errors $e^{\pm, 4}_k$, $\bar{e}^{(4)}_{\pm,k}$ and
$\tilde{e}^{(4)}_k$ given in \eqref{5.21},  \eqref{5.32} and \eqref{5.23} respectively, , we have
\begin{equation}\label{f3.85}
\begin{cases}
\|e^{\pm,4}_k\|_{s,X}\le C\delta^2\theta_k^{L_4(s)}\Delta_k, \quad s_0\le s\le s_1-2\cr
\|\overline{e}_{\pm,k}^{(4)}\|_
{s,X}\le C\delta^2\theta_k^{L_5(s-\frac 12)}\Delta_k, \quad s_0+\frac 32\le s\le s_1-\frac 72\cr
\|\tilde{e}_{k}^{(4)}\|_{H^s(\omega_X)}\le C\delta^2\theta_k^{L_5(s)}\Delta_k, \quad  s_0+1\le s\le s_1-4
\end{cases}
\end{equation}
for all $k\le n-1$, where
\begin{equation}\label{f3.86}
\begin{cases}
L_4(s)=\max((s+1-\alpha)_++2(s_0-\alpha+1), s+s_0+2-2\alpha),\cr
L_5(s)=\max((s+2-\alpha)_++2(s_0-\alpha)-1, s+s_0-2\alpha).
\end{cases}.
\end{equation}\end{lemma}

\begin{proof} (1) Denote by
\begin{equation}\label{f3.87}
R^\pm_k=\partial_y(L(U^{a,\pm}+V^{\pm}_{k+\frac 12},\Psi^{a,\pm}+S_{\theta_k}\Phi^{\pm}_{k})(U^{a,\pm}+V^{\pm}_{k+\frac 12})).
\end{equation}

Obviously, we have
$$\begin{array}{ll}
\|R^\pm_k\|_{s,X}\le &
\|L(U^{a,\pm}+V^{\pm}_{k+\frac 12},\Psi^{a,\pm}+S_{\theta_k}\Phi^{\pm}_k)(U^{a,\pm}+V^{\pm}_{k+\frac 12})\\[2mm]
&-
L(U^{a,\pm}+V^{\pm}_k,\Psi^{a,\pm}+\Phi^{\pm}_k)(U^{a,\pm}+V^{\pm}_k)\|_{s+1,X}\\[2mm]
&+
\|{\mathcal L}(V^{\pm}_k,\Phi^{\pm}_k)V^{\pm}_k-f_a^\pm\|_{s+1,X}
\end{array}
$$
which implies
\begin{equation}\label{f3.88}
\begin{array}{lll}
\|R^\pm_k\|_{s,X} &\le &
C\left(\|V^{\pm}_{k+\frac 12}-V^{\pm}_k\|_{W^{1,\infty}}
(\|\dot{U}^{a,\pm}+V^{\pm}_k\|_{s+1, X}+\|\dot{\Psi}^{a,\pm}+\Phi^{\pm}_k\|_{s+2, X})\right.\\[2mm]
& &+
\|V^{\pm}_{k+\frac 12}-V^{\pm}_k\|_{s+2, X}
(\|\dot{U}^{a,\pm}+V^{\pm}_k\|_{L^\infty}+\|\dot{\Psi}^{a,\pm}+\Phi^{\pm}_k\|_{W^{1,\infty}})\\[2mm]
& &+\|\dot{U}^{a,\pm}+V^{\pm}_{k+\frac 12}\|_{W^{1,\infty}}(
\|V^{\pm}_{k+\frac 12}-V^{\pm}_k\|_{s+1,X}+\|(1-S_{\theta_k})\Phi^\pm_k\|_{s+2,X})\\[2mm]
& &\left.+\|\dot{U}^{a,\pm}+V^{\pm}_{k+\frac 12}\|_{s+2,X}(
\|V^{\pm}_{k+\frac 12}-V^{\pm}_k\|_{L^\infty}+\|(1-S_{\theta_k})\Phi^\pm_k\|_{W^{1,\infty}})\right)\\[2mm]
& &+\|{\mathcal L}(V^{\pm}_k,\Phi^{\pm}_k)V^{\pm}_k-f_a^\pm\|_{s+1,X}\\[2mm]
\hspace{.3in}
& \le & C\delta^2(\theta_k^{(s+2-\alpha)_++s_0+1-\alpha}+\theta_k^{s+3-\alpha})
+2\delta \theta_k^{s-\alpha}
\end{array}
\end{equation}
for all $s_0\le  s\le s_1-3$. 

As in \cite{cou2}, as $s=s_1-2$, we immediately have
\begin{equation}\label{f3.88-1}
\|R^\pm_k\|_{s,X} \le \| L(U^{a,\pm}+V^{\pm}_{k+\frac 12},\Psi^{a,\pm}+S_{\theta_k}\Phi^{\pm}_k)(U^{a,\pm}+V^{\pm}_{k+\frac 12})\|_{s+1, X}\le C\delta \theta_k^{s+3-\alpha}.
\end{equation}

Thus, we get that
$$e^{\pm, 4}_{k}=\frac{R_k^\pm \delta\Phi^{\pm}_k}{\partial_{y}(\Psi^{a,\pm}+S_{\theta_k}\Phi^{\pm}_k)}$$
satisfy
\begin{equation}\label{f3.89}
\|e^{\pm,4}_k\|_{s,X}\le
C(\|R^\pm_k\|_{s_0,X}(\delta\theta_k^{s-1-\alpha}\Delta_k+\delta\theta_k^{s_0-1-\alpha}\Delta_k
(\delta+\delta\theta_k^{(s+1-\alpha)_+}))+\delta\theta_k^{s_0-1-\alpha}\Delta_k\|R_k^\pm\|_{s,X}),
\end{equation}
which yields the first estimate given in (\ref{f3.85}) for any $s_0\le s\le s_1-2$ by using
(\ref{f3.88}) and (\ref{f3.88-1}).

(2) Denote by
\begin{equation}\label{f3.90}
R^b_k={\mathcal B}(U^{a,\pm}+V^{\pm}_{k+\frac 12},\psi^a+S_{\theta_k}\phi_{k}).
\end{equation}
Obviously, we have
$$\begin{array}{ll}
\|R^b_k\|_{H^s(\omega_X)}\le
& \|{\mathcal B}(U^{a,\pm}+V^{\pm}_{k+\frac 12},\psi^a+S_{\theta_k}\phi_{k})-{\mathcal B}(U^{a,\pm}+V^{\pm}_k,\psi^a+\phi_{k})\|_{H^s(\omega_X)}\\[2mm]
&
+\|{\mathbb B}(V^{\pm}_k,\phi_{k})\|_{H^s(\omega_X)},
\end{array}
$$
which implies
\begin{equation}\label{f3.91}
\begin{array}{lll}
\|(R^b_k)_1\|_{H^s(\omega_X)}& \le &
C\left(
\|(S_{\theta_k}-1)\phi_{k}\|_{H^{s+1}(\omega_X)}\|\dot{U}^{a,+}+V_k^+\|_{L^\infty}
+\|(S_{\theta_k}-1)\phi_{k}\|_{W^{1,\infty}}\|\dot{U}^{a,+}+V_k^+\|_{H^{s}(\omega_X)}\right.\\[2mm]
& &
\left.+\|(S_{\theta_k}-1)V^+_{k}\|_{H^{s}(\omega_X)}\|\dot{\psi}^{a,+}+S_{\theta_k}\phi_k\|_{W^{1,\infty}}
+\|(S_{\theta_k}-1)V^+_{k}\|_{L^\infty}\|\dot{\psi}^{a,+}+S_{\theta_k}\phi_k\|_{H^{s+1}(\omega_X)}\right)\\[2mm]
& & +\|{\mathbb B}(V^+_k,V^-_k,\phi_{k})\|_{H^s(\omega_X)}\\[2mm]
& \le &
C\delta\theta_k^{\max((s-\alpha)_++s_0-\alpha, ~s-\alpha)}
\end{array}
\end{equation}
for all $s_0\le s\le s_1-3$.

Thus, we get that
$$\tilde{e}_{k, 1}^{(4)}=-\frac{\partial_{y}(R_k^b)_1}{\partial_{y}(\Psi^{a,\pm}+S_{\theta_k}\Phi^{\pm}_k)|_{y=0}} \delta\phi_k$$
satisfies
\begin{equation}\label{f3.93}
\|\tilde{e}_{k,1}^{(4)}\|_{H^s(\omega_X)}\le
C\left(\|(R^b_k)_1\|_{H^3(\omega_X)}(\delta\theta_k^{s-\alpha-2}\Delta_k+\delta^2\theta_k^{(s+2-\alpha)_++s_0-\alpha-1}\Delta_k)+\delta\theta_k^{s_0-\alpha-1}\Delta_k\|(R^b_k)_1\|_{H^{s+1}(\omega_X)}\right),
\end{equation}
which yields the estimate of $\tilde{e}_{k,1}^{(4)}$ given in (\ref{f3.85})
for any $s_0+1\le s\le s_1-4$ by using
(\ref{f3.91}).

One can obtain the estimate (\ref{f3.85}) of other components of $\tilde{e}_{k}^{(4)}$ similarly.

Noting that the trace of $\overline{e}_{\pm,k}^{(4)}$ on $\{y=0\}$ is equal to  $(\tilde{e}_{k,1}^{(4)}, \tilde{e}_{k,2}^{(4)})$,
the estimate of $\overline{e}_{\pm,k}^{(4)}$ in (\ref{f3.85}) follows immediately. \end{proof}

\vspace{.1in}

Summarizing all results from Lemmas \ref{lemma6.2}, \ref{lemma6.3}, \ref{lemma6.5} and \ref{lemma6.6}, we conclude

\begin{lemma} \label{lemma6.7}
The errors $e^\pm_{k}=\sum_{j=1}^4e^{\pm, j}_{k},  \overline{e}^{\pm}_k=\sum_{j=1}^4\overline{e}^{(j)}_{\pm, k}$ and $\tilde{e}_{k}=\sum_{j=1}^4\tilde{e}^{ (j)}_{k}$  satisfy
\begin{equation}\label{f3.94}
\begin{cases}
\|e^\pm_{k}\|_{s,X}\le C\delta^2\theta_k^{L_4(s)}\Delta_k, \quad s_0\le s\le s_1-2\cr
\|\overline{e}^{\pm}_k\|_
{s,X}\le C\delta^2\theta_k^{L_6(s)}\Delta_k, \quad  s_0+\frac 32\le s\le s_1-\frac 72\cr
\|\tilde{e}_{k}\|_{H^s(\omega_X)}\le C\delta^2\theta_k^{L_5(s)}\Delta_k, \quad   s_0+1 \le s\le s_1-4
\end{cases}
\end{equation}
for all $k\le n-1$, where $L_4(s)$ and $L_5(s)$ are given in Lemma 6.6, and 
\begin{equation}\label{ll}
L_6(s)=\max((s+{\frac 32}-\alpha)_++2(s_0-\alpha)-1, s+s_0-2\alpha).
\end{equation}

\end{lemma}

\vspace{.1in}
From Lemma \ref{lemma6.7}, we immediately obtain

\begin{lemma}\label{lemma6.8}
 For any fixed $s_0>\frac 52$, $\alpha\ge s_0+2$ and $s_0+3 \le s_1\le 2\alpha-s_0+1$, the accumulated errors
\begin{equation}\label{f3.96}
E^\pm_n=\sum_{k=0}^{n-1}e^\pm_{k}, \quad \tilde{E}_{n}=\sum_{k=0}^{n-1}\tilde{e}_{k},
\quad \overline{E}^\pm_{n}=\sum_{k=0}^{n-1}\overline{e}^\pm_{k}
\end{equation}
satisfy the estimates
\begin{equation}\label{f3.97}
\begin{cases}
\|E^{\pm}_n\|_{s,X}\le C\delta^2\theta_n, \quad s_0\le s\le s_1-2,\cr
\|\overline{E}^\pm_n\|_
{s,X}\le C\delta^2\theta_n, \quad s_0+\frac 32\le s\le s_1-\frac 72,\cr
\|\tilde{E}_{n}\|_{H^s(\omega_X)}\le C\delta^2\theta_n, \quad  s_0+1\le s\le
s_1-4.\end{cases}.
\end{equation}\end{lemma}

\vspace{.1in}
To study problems (\ref{5.19}) and (\ref{5.27})-(\ref{5.28}), first we have

\begin{lemma}\label{lemma6.9}
With the same range of $s_0$ and  $s_1$ as given in Lemma \ref{lemma6.8},  we have
\begin{equation}\label{f3.98}
\begin{cases}
\|f_n^\pm\|_{s,X}\le C\Delta_n (\theta_n^{s-s_2-1}\|f_a^\pm\|_{s_2,X}+\delta^2\theta_n^{s-s_3}+\delta^2\theta_n^{(s-s_4)_++L_4(s_4)})\cr
\|g_n\|_{H^s(\omega_X)}\le C\delta^2\Delta_n(\theta_n^{s-s_5}+\theta_n^{(s-s_6)_++L_5(s_6)})\cr
\|h^\pm_n\|_{s,X}\le C\delta^2\Delta_n (\theta_n^{s-s_7}+\theta_n^{(s-s_8)_++L_6(s_8)+L_5(s_8-\frac 12)})\cr
\end{cases}
\end{equation}
for all
$$\begin{cases}
s_2\ge 0,\quad
s_0\le s_3, \; s_4\le s_1-2,\\[2mm]
  s_0+1\le s_5,\; s_6\le s_1-4,\\[2mm]
s_0+\frac 32\le s_7, \;s_8\le s_1-\frac 72,
\end{cases}
$$
where $L_4(s)$, and $L_6(s)$, $L_7(s)$ are given in \eqref{f3.86} and \eqref{ll} respectively.
\end{lemma}

\begin{proof}
From the definitions of $f_n^\pm, g_n$ and $h_n^\pm$ given in (\ref{f-n}), (\ref{g-n}) and \eqref{5.34}-\eqref{5.35} respectively, obviously we have
\begin{equation}\label{f3.99}
\begin{cases}
f_n^\pm=(S_{\theta_n}-S_{\theta_{n-1}})f_a^\pm-(S_{\theta_n}-S_{\theta_{n-1}})E^\pm_{n-1}
-S_{\theta_{n}}e^\pm_{n-1},\cr
g_n=-(S_{\theta_n}-S_{\theta_{n-1}})\tilde{E}_{n-1}-S_{\theta_{n}}\tilde{e}_{n-1},\cr
h_n^+=-(S_{\theta_{n}}-S_{\theta_{n-1}})(\bar{E}^+_{n-1}-{\mathbb E}(\tilde{E}_{n-1,1}))
-S_{\theta_{n}}(\overline{e}_{+,n-1}-{\mathbb E}(\tilde{e}_{n-1,1})),\cr
h_n^-=-(S_{\theta_{n}}-S_{\theta_{n-1}})(\bar{E}^-_{n-1}-{\mathbb E}(\tilde{E}_{n-1,2}))
-S_{\theta_{n}}(\overline{e}_{-,n-1}-{\mathbb E}(\tilde{e}_{n-1,2})).
\end{cases}
\end{equation}

By using the properties of the smoothing operators, and Lemmas \ref{lemma6.7} and \ref{lemma6.8} we have that for all $s\ge 0$,
$$\begin{cases}
\|(S_{\theta_n}-S_{\theta_{n-1}})f_a^\pm\|_{s,X}\le
C\theta_n^{s-\tilde{s}-1}\Delta_n\|f_a^\pm\|_{\tilde{s},X}, \quad \tilde{s}\ge 0\cr
\|(S_{\theta_n}-S_{\theta_{n-1}})E^\pm_{n-1}\|_{s,X}\le
C\theta_n^{s-\tilde{s}-1}\Delta_n\|E^\pm_{n-1}\|_{\tilde{s},X}\le C\delta^2\theta_n^{s-\tilde{s}}\Delta_n, \quad s_0\le \tilde{s}\le s_1-2\cr
\|S_{\theta_{n}}{e}^\pm_{n-1}\|_{s,X}\le
C\theta_n^{(s-\tilde{s})_+}\|e^\pm_{n-1}\|_{\tilde{s},X}\le C \delta^2\theta_n^{(s-\tilde{s})_++L_4(\tilde{s})}\Delta_n,
 \quad s_0\le \tilde{s}\le s_1-2
\end{cases}
$$
which  implies the first estimate given in (\ref{f3.98}).

Similarly, we have
$$\begin{cases}
\|(S_{\theta_n}-S_{\theta_{n-1}})\tilde{E}_{n-1}\|_{H^s(\omega_X)}\le
C\theta_n^{s-\tilde{s}-1}\Delta_n\|\tilde{E}_{n-1}\|_{H^{\tilde{s}}(\omega_X)}
\le C\delta^2\theta_n^{s-\tilde{s}}\Delta_n  \cr
\|S_{\theta_{n}}\tilde{e}_{n-1}\|_{H^s(\omega_X)}\le
C\theta_n^{(s-\tilde{s})_+}\|\tilde{e}_{n-1}\|_{H^{\tilde s}(\omega_X)}\le
C\delta^2\theta_n^{(s-\tilde{s})_++L_5(\tilde{s})}\Delta_n
\end{cases}
$$
for all $s_0+1\le \tilde{s}\le
s_1-4$, this follows the estimate of $g_n$ given in (\ref{f3.98}) immediately.

From the definition of $h_n^\pm$, we have
$$\|h_n^\pm\|_{s, X}\le
C\theta_{n}^{s-\tilde{s}-1}\Delta_n\|\overline{E}^\pm_{n-1}-{\mathbb E}(\tilde{E}_{n-1})\|_{\tilde{s}, X}
+C\theta_{n}^{(s-\tilde{s})_+}\|\overline{e}^\pm_{n-1}-{\mathbb E}(\tilde{e}_{n-1})\|_{\tilde{s}, X},$$
which yields the conclusion given in (\ref{f3.98}) by  using Lemmas \ref{lemma6.7} and \ref{lemma6.8}. \end{proof}

\vspace{.1in}

To close this Nash-Moser iteration scheme, it remains to verify the inductive assumption
($H_{n}$) given at the beginning of this section.

\vspace{.1in}
\underline{\bf Verification of the assumption ($H_{n}$).}

\vspace{.1in}
(1) The assertion of ($H_0$) can be easily verified by studying the problem \eqref{5.19} with $n=0$ of $\delta\tilde{V}^\pm_0$,  and the problem \eqref{5.27} with $n=0$ of $\delta\Phi^\pm_0$.

(2) Assume that ($H_{n}$) holds, let us study ($H_{n+1}$).

To apply Corollary \ref{coro4.7} in the problem (\ref{5.19}), first we note that
$$
\begin{array}{l}
\|\dot{U}^{a,\pm}+V^{\pm}_{n+\frac 12}, \nabla(\dot{\Psi}^{a,\pm}+
S_{\theta_n}\Phi^{\pm}_{n})\|_{s_0+2, X}\le
\|\dot{U}^{a,\pm}, \nabla \dot{\Psi}^{a,\pm}\|_{s_0+2, X}\\[2mm]
\hspace{.5in}
+\|V^{\pm}_{n+\frac 12}-S_{\theta_n}V^{\pm}_{n}\|_{s_0+2, X}+
\|S_{\theta_n}V^{\pm}_{n},\nabla(S_{\theta_n}\Phi^{\pm}_{n})\|_{s_0+2,X}\\[2mm]
\hspace{.5in}\le
C\delta(1+\theta_n^{s_0+3-\alpha}+\theta_n^{(s_0+3-\alpha)_+})
\le C\delta
\end{array}$$
when $\alpha\ge  s_0+3$, by using the assumption (\ref{6.1}) and Lemmas \ref{lemma6.1} and \ref{lemma6.4}.

Thus, we can apply Corollary \ref{coro4.7} in the problem  (\ref{5.19}) to obtain
\begin{equation}\label{f3.100}
\begin{array}{l}
\|\delta\tilde{V}^{\pm}_{n}\|_{s,X}+
\|\delta\phi_{n}\|_{H^{s+1}(\omega_X)}
\le
C(\|f_n^\pm\|_{s+1,X}+\|g_n\|_{H^{s+1}(\omega_X)}\\[2mm]
\hspace{.8in}+\|\dot{U}^{a,\pm}+V^{\pm}_{n+\frac 12},
\nabla(\dot{\Psi}^{a,\pm}+S_{\theta_n}\Phi^{\pm}_{n})\|_{s+2,X}(\|f_n^\pm\|_{s_0+2,X}+\|g_n\|_{H^{s_0+2}(\omega_X)})).
\end{array}
\end{equation}
When $\alpha\ge s_0+4$ and $s_1\ge \alpha+5$, setting $s_2=\alpha+1$ and $s_3=\alpha+2$ in Lemma \ref{lemma6.9}, it follows
$$\|f_n^\pm\|_{s+1,X}\le C\delta \theta_n^{s-\alpha-1}\Delta_n(\frac{\|f_a^\pm\|_{\alpha+1,X}}{\delta}+\delta)+C\delta^2\Delta_n\theta_n^{(s+1-s_4)_++L_4(s_4)}.
$$

On the other hand, by setting
$$s_4=\begin{cases}
s, \quad s_0\le s\le s_1-2\cr
s_1-1, \quad s_1-1\le s\le s_1
\end{cases}$$
one has
$$(s+1-s_4)_++L_4(s_4)\le s-\alpha-1$$
for all $s_0\le s\le s_1$, thus we get
\begin{equation}\label{6.5}
\|f_n^\pm\|_{s+1,X}\le C\delta \theta_n^{s-\alpha-1}\Delta_n(\frac{\|f_a^\pm\|_{\alpha+1,X}}{\delta}+\delta)+C\delta^2\Delta_n\theta_n^{s-\alpha-1}.
\end{equation}

Similarly, as $s_1\ge \alpha+6$, by setting $s_5=\alpha+2$, and
$$s_6=\begin{cases}
s+1, \quad s_0\le s\le s_1-5\cr
s_1-4, \quad s_1-4\le s\le s_1
\end{cases}$$
in Lemma \ref{lemma6.9}, we have
\begin{equation}\label{6.6}
\|g_n\|_{H^{s+1}(\omega_X)}\le C\delta^2\Delta_n\theta_n^{s-\alpha-1}
\end{equation}
for all $s_0\le s\le s_1$.

When $\alpha\ge s_0+5$, by letting $s_2=\alpha+2$, $s_3=s_5=\alpha+3$, $s_4=s_6=\alpha-3$ in Lemma \ref{lemma6.9}, we have
$$ \|f_n^\pm\|_{s_0+2,X}+\|g_n\|_{H^{s_0+2}(\omega_X)}\le
C\delta\theta_n^{s_0-\alpha-1}\Delta_n$$
provided that
$$\frac{\|f_a^\pm\|_{\alpha+2,X}}{\delta}\le C<+\infty.$$

On the other hand, from the assumption (\ref{6.1}) and Lemmas \ref{lemma6.1} and \ref{lemma6.4} we have
\begin{equation}\label{f3.102}
\|\dot{U}^{a,\pm}+V^{\pm}_{n+\frac 12},
\nabla
(\dot{\Psi}^{a,\pm}+S_{\theta_n}\Phi^{\pm}_n)\|_{s+2,X}\le
C\delta(1+\theta_n^{s+3-\alpha}+\theta_n^{(s+3-\alpha)_+})
\end{equation}
which implies
\begin{equation}\label{f3.103}
\|\dot{U}^{a,\pm}+V^{\pm}_{n+\frac 12},
\nabla(\dot{\Psi}^{a,\pm}+S_{\theta_n}\Phi^{\pm}_n)\|_{s+2,X}
(\|f_n^\pm\|_{s_0+2,X}+\|g_n\|_{H^{s_0+2}(\omega_X)})\le
C\delta^2\theta_n^{s-\alpha-1}\Delta_n
\end{equation}
for all $s_0\le s\le s_1$.

Thus, plugging \eqref{6.5}, \eqref{6.6} and \eqref{f3.103} into (\ref{f3.100}) it follows
\begin{equation}\label{f3.104}
\|\delta \tilde{V}^{\pm}_n\|_{s,X}
+\|\delta \phi_n\|_{H^{s+1}(\Omega_X)}
\le
C\delta \theta_n^{s-\alpha-1}\Delta_n(\frac{\|f_a^\pm\|_{\alpha+1,X}}{\delta}+\delta)+C\delta^2\Delta_n\theta_n^{s-\alpha-1}
\end{equation}
for all $s_0\le s\le s_1$.

For the problems \eqref{5.27} and \eqref{5.28}, one can easily deduce the following estimate
\bel{6.10}
\begin{split}
&\|\delta\Phi^\pm_n\|_{s,X}
\leq
C\left\{\|g_n\|_{H^{s-\frac 12}(\omega_X)}
+\|h_n^\pm\|_{s,X}+(1+\|\nabla(\Psi^{a,\pm}+S_{\theta_n}\Phi_n^\pm)\|_{s_0,X})\|\delta \tilde{V}^{\pm}_n\|_{s,X}\right.\\[2mm]
&\hspace{.6in}\left.+\|\nabla(\Psi^{a,\pm}+S_{\theta_n}\Phi_n^\pm)\|_{s,X}\|\delta \tilde{V}^{\pm}_n\|_{s_0,X}
\right\}
\end{split}
\eeq
for all $s\ge 0$.

Similar to the discussion for the estimate \eqref{6.6} of $g_n$, by choosing $s_7, s_8$ properly in Lemma \ref{lemma6.9}, we can get
$$\|h_n^\pm\|_{s, X}\le C\delta^2\Delta_n\theta_n^{s-\alpha-1}, \quad s_0\le s\le s_1.$$

Thus, by using ($H_{n}$), \eqref{6.6} and \eqref{f3.104} in \eqref{6.10}, it follows
\begin{equation}\label{6.11}
\|\delta\Phi^\pm_n\|_{s,X}
\le
C\delta \theta_n^{s-\alpha-1}\Delta_n(\frac{\|f_a^\pm\|_{\alpha+1,X}}{\delta}+\delta)+C\delta^2\Delta_n\theta_n^{s-\alpha-1}
\end{equation}
for all $s_0\le s\le s_1$.

Together (\ref{f3.104}) with (\ref{6.11}), it follows ($H_{n+1}$) for $\delta V^{\pm, n}$ and $\delta \Phi^{\pm, n}$ immediately
by using
$$\delta V^{\pm}_n=\delta\tilde{V}^{\pm}_n+\frac{\partial_{y}(U^{a,\pm}+V^{\pm}_{n+\frac 12})}
{\partial_{y}(\Psi^{a,\pm}+ S_{\theta_n}\Phi^{\pm}_n)}\delta\Phi^{\pm}_n
$$
and letting both of $\delta$, $\frac{\|f_a^\pm\|_{\alpha+1,X}}{\delta}$ being properly small.

To verify other inequalities in ($H_{n+1}$), we shall use the idea from \cite{cou2}. From (\ref{5.21}),
we have
\begin{equation}\label{f3.107}
{\mathcal L}(V^{\pm}_{n+1}, \Phi^{\pm}_{n+1})V^{\pm}_{n+1}-f_a^\pm
=(S_{\theta_{n}}-I)f_a^\pm+(I-S_{\theta_{n}})E^{\pm}_{n}+
e^\pm_{ n}.
\end{equation}

From Lemma \ref{lemma6.8}, we have
\begin{equation}\label{f3.108}
\|(I-S_{\theta_{n}})E^{\pm}_{n}\|_{s,X}\le C\theta_n^{s-\tilde{s}}\|E^\pm_{n}\|_{\tilde{s},X}
\le C\theta_n^{s-\tilde{s}+1}\delta^2\le C\theta_n^{s-\alpha-1}\delta^2
\end{equation}
for all $s\le s_1-2$ by choosing $\tilde{s}=s_1-2$.

As we already have estimates of $\delta V^\pm_n$, $\delta\Phi_n^\pm$ given in ($H_{n+1}$), the result of Lemma \ref{lemma6.7} is also true for $k=n$, thus we have
\begin{equation}\label{f3.109}
\|e^\pm_{n}\|_{s,X}
\le C\theta_n^{L_4(s)}\delta^2\Delta_n\le C\theta_n^{s-\alpha-2}\delta^2\Delta_n
\end{equation}
for all $s_0\le s\le s_1-2$.

On the other hand, it is  easy to have
\begin{equation}\label{f3.110}
\begin{cases}
\|(S_{\theta_{n}}-I)f_a^\pm\|_{s,X}\le C\theta_n^{s-\alpha-1}\|f_a^\pm\|_{\alpha+1, X}, \quad s\le \alpha+1\cr
\|(S_{\theta_{n}}-I)f_a^\pm\|_{s,X}\le \|S_{\theta_{n}}f_a^\pm\|_{s,X}+\|
f_a^\pm\|_{s,X}\le C\theta_n^{s-\alpha-1}\|f_a^\pm\|_{\alpha+1, T}+C\delta, \quad \alpha+2\le s\le s_1+1.
\end{cases}
\end{equation}

Substituting  (\ref{f3.110}), (\ref{f3.109}) and (\ref{f3.108}) into (\ref{f3.107}),
it follows that
\begin{equation}\label{f3.112}
\|{\mathcal L}(V^{\pm}_{n+1}, \Phi^{\pm}_{n+1})V^{\pm}_{n+1}-f_a^\pm\|_{s,X}\le \delta \theta_n^{s-\alpha-1}
\end{equation}
for all  $s_0\le s\le s_1-2$.

Similarly, one can verify the last assertion of ($H_{n+1}$) for the estimate of $
{\mathbb B}(V^{+}_{n+1},V^{-}_{n+1}, \phi_{n+1})$.


\end{document}